\documentclass[12pt, reqno]{amsart}
\usepackage{amsthm}
\usepackage{amsfonts}
\usepackage{xr}
\usepackage{graphicx}
\usepackage{amsmath}
\usepackage{epsfig,subfigure}
\usepackage{color}
\usepackage{amssymb}

\def\Z{\mathbb{Z}}
\def\R{\mathbb{R}}
\def\N{\mathbb{N}}

\def\epsilon{\varepsilon}

\def\tilde{\widetilde}

\newcommand{\SE}{\setcounter{equation}{0} \section}
\newcommand{\be}{\begin{equation}}
\newcommand{\ee}{\end{equation}}
\newcommand{\baa}{\begin{array}}
\newcommand{\eaa}{\end{array}}
\newcommand{\ba}{\begin{eqnarray}}
\newcommand{\ea}{\end{eqnarray}}
\numberwithin{equation}{section}

\newtheorem{theorem}{\bf Theorem}[section]
\newtheorem{lemma}[theorem]{\bf Lemma}
\newtheorem{prop}[theorem]{\bf Proposition}
\newtheorem{corollary}[theorem]{\bf Corollary}

\newtheorem{remark}[theorem]{\bf Remark}

\begin{document}
\date{}
\title[Dynamics of time-periodic reaction-diffusion equations]{Dynamics of time-periodic reaction-diffusion equations with compact initial support on $\R$} 

\thanks{This work was supported by the Japan Society for the Promotion of Science (16H02151, 17F17021).
Part of this work was carried out during the authors worked in the Graduate School of Mathematical Sciences,  University of Tokyo, whose support 
is thankfully acknowledged.}

\author[W. Ding and H. Matano]{Weiwei Ding$^\dag$ and Hiroshi Matano$^\ddag$}
\thanks{
$^\dag$ Meiji Institute for Advanced Study of Mathematical Sciences, Meiji University, Tokyo 164-8525, Japan (e-mail: dingww@mail.ustc.edu.cn)}

\thanks{
$^\ddag$ Meiji Institute for Advanced Study of Mathematical Sciences, Meiji University, Tokyo 164-8525, Japan (e-mail: matano@meiji.ac.jp)}


\begin{abstract}
This paper is concerned with the asymptotic behavior of bounded solutions of the Cauchy problem
\begin{equation*}
\left\{\baa{ll}
\smallskip u_t=u_{xx} +f(t,u), & x\in\R,\,t>0, \vspace{3pt}\\
u(x,0)= u_0, & x\in\R,\eaa\right.
\end{equation*}
where $u_0$ is a nonnegative bounded function with compact support and $f$ is a rather general nonlinearity that is periodic in $t$ and satisfies $f(\cdot,0)=0$. In the autonomous case where $f=f(u)$, the convergence of every bounded solution to an equilibrium has been established by Du and Matano (2010). However, the presence of periodic forcing makes the problem significantly more difficult, partly because the structure of time periodic solutions is much less understood than that of steady states. In this paper, we first prove that any $\omega$-limit solution is either spatially constant or symmetrically decreasing, even if the initial data is not symmetric. Furthermore, we show that the set of $\omega$-limit solutions either consists of a single time-periodic solution or it consists of multiple time-periodic solutions and heteroclinic connections among them. Next, under a mild non-degenerate assumption on the corresponding ODE, we prove that the $\omega$-limit set is a singleton, which implies the solution converges to a time-periodic solution. Lastly, we apply these results to equations with bistable nonlinearity and combustion nonlinearity, and specify more precisely which time-periodic solutions can possibly be selected as the limit. 
\end{abstract}

\maketitle


\SE{Introduction and main results}
In this paper, we study the long-time behavior of nonnegative bounded solutions of 
the Cauchy problem
\begin{subequations}\label{E}
\begin{eqnarray}
&u_t=u_{xx}+f(t,u), \ \ \ &x\in\R,\,\,t>0,
\label{equation}\\[4pt]
&u(x,t)=u_0(x),\ \ \ &x\in\R,
\label{initial}
\end{eqnarray}
\end{subequations}
where the initial data $u_0\in L^{\infty}(\R)$ is nonnegative. The nonlinearity $f:\R\times [0,\infty)\to\R$
is locally H{\"o}lder continuous in $\R\times[0,\infty)$, and it is of class $C^1$ with respect to $u$. 
We also assume that
\begin{equation}\label{0state}
f(t,0)=0\,\,\hbox{ for all } t\in\R,
\end{equation}
and that $f$ is $T$-periodic in $t$ for some $T>0$, that is, 
\begin{equation}\label{tperiod}
f(t+T,u)=f(t,u) \,\,\hbox{ for all } t\in\R,\,u\geq 0.
\end{equation}

It is well-known that, for each bounded nonnegative $u_0$, \eqref{E} admits a unique local-in-time solution $u(x,t)$ satisfying 
$$\sup_{t\in (0,\delta]} \|u(\cdot,t) \|_{L^{\infty}(\R)}<\infty,\qquad \lim_{t\to 0} u(x,t)=u_0(x) \, \hbox{ for a.e. } x\in\R,  $$ 
for some $\delta>0$. Moreover, the solution $u(x,t)$ is smooth in $t>0$, and it is defined as long as it stays finite. 
Denote by $T(u_0)$ the maximum time-interval for the existence of $u(x,t)$. Clearly, if $u(x,t)$ is bounded on $\R\times [0,T(u_0))$, then necessarily $T(u_0)=\infty$. We are interested in the asymptotic behavior of bounded solutions as $t\to \infty$.   
 
\vskip 5pt

Before going further, let us give a brief overview of some known results on the asymptotic behavior of bounded solutions of \eqref{E}. For equation \eqref{E} on a finite interval under various types of boundary conditions (Dirichlet, Neumann, Robin, Periodic), it is known that every bounded solution converges to an equilibrium in the autonomous case (i.e. $f=f(u)$ or even $f=f(x,u)$), see \cite{c1,hr,m,z68};  it converges to a time-periodic solution in the time-periodic case (i.e. $f=f(t,u)$ with \eqref{tperiod}), see \cite{bps,cm}. These results have been extended to higher dimensions under certain symmetry assumptions, see \cite{cp,hap,hp} and references thererin.   
 
For problem \eqref{E} on the entire real line $\R$, the situation is more complicated, even for the autonomous equation  
\begin{equation}\label{homo-equation}
u_t=u_{xx}+f(u),\,\,x\in\R,\,t>0.
\end{equation} 
Indeed, there are examples of bounded solutions of \eqref{homo-equation} that oscillate between two equilibria forever; see \cite{er,po2,po3}. The presence of such phenomena makes sharp contrast between the problems on a finite interval and those on the entire line $\R$.
However, if the initial function $u_0$ is nonnegative and compactly supported, one can prove that such oscillation does not occur, and therefore that any bounded solution converges to an equilibrium in the topology of $L_{loc}^{\infty}(\R)$. Moreover, the limit equilibrium is either spatially constant (hence it is a zero of $f$) or it is symmetrically decreasing to a zero of $f$ as $|x|\to\infty$. The most general results in this direction was established in \cite{dm}, where $f$ is only assumed to be locally Lipschitz continuous and satisfy $f(0)=0$. Let us also mention an earlier work \cite{z1}, in which similar results are obtained for special classes of $f$ and $u_0$.
Incidentally, the paper \cite{dm,z1} also discuss the sharpness of transition between propagation and extinction when $f$ is of the bistable or the combustion type nonlinearity; see also Remark \ref{sharp-transition} below.
It should be pointed out that the compactness of the support of the initial data is an essential assumption in \cite{dm,z1}. The above convergence result was extended in \cite{dp} to higher dimensions under some additional non-degeneracy assumptions on $f$. On the other hand, the problem becomes much harder if one simply assumes that initial data $u_0$ satisfies $u_0(x)\to 0$ (as $|x|\to\infty$) instead of being compactly supported. This case was studied in \cite{mp}, in which a general quasiconvergence result (the $\omega$-limit set consists of equilibria) was established.

In this paper, we focus our attention on problem \eqref{E} with time-periodic nonlinearity $f$, and assume that the 
initial function $u_0$ is nonnegative and compactly supported. Based on the aforementioned convergence results, it may be reasonable to expect that every bounded solution of \eqref{E} converges to a time-periodic solution of \eqref{equation} in $L_{loc}^{\infty}(\R)$. However, the presence of periodic forcing makes the problem significantly more difficult than in the autonomous case. Indeed, the proof given in \cite{dm} relies for a large part on the zero-number arguments (also known as intersection comparison principle) and on the classification of all steady states of \eqref{homo-equation}.  While the zero-number arguments is still a powerful tool in our time-periodic problem, a complete understanding of the structure of all time-periodic solutions of \eqref{equation} may be beyond reach. Besides, the method of energy functional used in many autonomous cases (see e.g., \cite{fa,fe}) to conclude the convergence does not apply to our non-autonomous problem. 
 
To our knowledge, in this time-periodic case, only few partial results on the asymptotic behavior of bounded solutions have been known in the literature. Those results are given either for a particular type of nonlinearities or under a strong assumption on the behavior of solutions at $|x|=\infty$. To be more precise, it was proved in \cite{fp} that every bounded solution $u(x,t)$ converges in the topology of $C_0(\R)$ (the space of continuous functions on $\R$, not necessarily nonnegative, converging to $0$ as $|x|\to\infty$), under the following assumptions, in addition to \eqref{0state} and \eqref{tperiod}: 
 \begin{equation}\label{deri-negative}
 \partial_u f(t,0)<0 \,\,\hbox{ for } t\in\R,
 \end{equation}
 and 
 \begin{equation}\label{implicit-condition}
 \big\{ u(\cdot,t):t\in [0,\infty)  \big\}  \hbox{ is relatively compact in } C_0(\R).
 \end{equation}
The proof in \cite{fp} uses the zero-number arguments and the properties of invariant manifolds in a neighborhood of the equilibria. However, the restriction \eqref{implicit-condition} rules out the cases where the solution converges to an equilibrium not decaying to zero as $|x|\to\infty$, which is a situation commonly observed in many applications. In a more recent paper \cite{po1}, Pol{\'a}{\v c}ik considered  
the case where $f(t,u)$ can be bounded from above and below by autonomous bistable nonlinearities, but allowing $f$ to be even non-periodically dependent on $t$, and established a result on the sharp transition between propagation and extinction as an extension of the results in \cite{dm,z1} for the autonomous bistable case $f=f(u)$. When one applies this result to the time-periodic case, it implies that every bounded solution with nonnegative and compactly supported initial function converges in $L_{loc}^{\infty}(\R)$ to a time-periodic solution.

\vskip 5pt

In the present paper, for general time-periodic nonlinearity $f$, we first give a precise description of the asymptotic behavior of nonnegative bounded solutions. 
In order to formulate our main results, let us introduce some notations. Throughout this paper, we always use $u(x,t)$ to denote the solution of \eqref{E} and assume that it is bounded. 
Its $\omega$-limit set is denoted by 
\begin{equation*}
\omega(u):=\Big\{\phi :\, u(x,k_j T) \to \phi(x) \hbox{ for some sequence of positive intergers }  k_j\to\infty \Big\},
\end{equation*}
with the convergence understood in the topology of $L^\infty_{loc}(\R)$.  By parabolic estimates, this convergence also takes place in the $C^{2}_{loc}(\R)$ topology. Elements of $\omega(u)$ are called {\bf 
$\omega$-limit points} of $u$.  For each 
$\phi\in\omega(u)$, there exists an entire solution 
$U(x,t)$ of the equation 
\begin{equation}\label{ground}
U_t=U_{xx} +f(t,U) \,\,\hbox{ for } x\in\R,\,t\in\R
\end{equation}
that satisfies $U(x,0)=\phi(x)$. We call such $U$ an 
{\bf $\omega$-limit solution}  through $\phi$, and denote the set of all $\omega$-limit solutions by $\tilde{\omega}(u)$. It is clear that 
\begin{equation*}
\omega(u)= \Big\{U(\cdot,mT): m\in\Z,\, U\in \tilde{\omega}(u) \Big\}.
\end{equation*}

A solution $U(x,t)$ of \eqref{ground} is said to be {\bf symmetrically decreasing} with respect to some $x_0\in\R$ if $U(2x_0-x,t)=U(x,t)$ for all $x\in\R$, $t\in\R$, and $\partial _x U (x,t)<0$ for all $x>x_0$, $t\in\R$. Since we only consider nonnegative solutions of \eqref{ground}, such a symmetrically decreasing solution $U(x,t)$ has a limit 
$$p(t):=\lim_{|x|\to\infty} U(x,t)\,\,\hbox{ for } t\in\R.   $$
It is clear that $p(t)$ is a solution of 
\begin{equation}\label{ODE}
p_t=f(t,p)\, \,\hbox{ for } t\in\R.
\end{equation}
In this paper, we call $p(t)$ {\bf the base} of $U(x,t)$. It is also easily seen that if $U(x,t)$ is $T$-periodic in $t$, then $p(t)$ is $T$-periodic.

\vskip 5pt

Our first main result is stated as follows. 

\begin{theorem}\label{mainconv}
Assume that $u_0\in L^{\infty}(\R)$ is a nonnegative function with compact support. 
Let $f$ satisfy \eqref{0state}, \eqref{tperiod} and assume further that
\begin{equation}\label{assume-lipschitz}
\partial_u f(t,u) \hbox{ is locally Lipschitz continuous in $u$ uniformly for $t$}.
\end{equation}
 If the solution $u(x,t)$ of \eqref{E} is bounded, then there exists $x_0\in\R$ such that every element of $\tilde{\omega}(u)$ is either spatially constant or symmetrically decreasing with respect to $x_0$.
Moreover, one of the following holds:
\begin{itemize}
\item[{\rm (i)}] $\tilde{\omega}(u)$ consists of only one element, 
and it is either a $T$-periodic solution of \eqref{ODE} or 
a $T$-periodic symmetrically decreasing solution of \eqref{ground};
\item[{\rm (ii)}] $\tilde{\omega}(u)$ contains more than one element, and all of them 
are symmetrically decreasing with respect to $x_0$ and they have the same base. In addition, each $v\in \tilde{\omega}$ is either a $T$-periodic solution of \eqref{ground} or a heteroclinic solution connecting
two $T$-periodic solutions $V_{\pm}(x,t)$, i.e., for any $t\in\R$, there holds
\begin{equation}\label{heteroclinic}
v(x,t+mT)-V_{\pm}(x,t)\to 0 \,\,\,\hbox{ as } \,\,m\to\pm \infty \hbox{ in }  L^{\infty}(\R). 
\end{equation}
\end{itemize}
\end{theorem}

The above theorem implies that $\tilde{\omega}(u)$ either consists in precisely one periodic solution or it consists of multiple symmetrically decreasing periodic solutions and heteroclinic connections among them. 
It also immediately implies the following corollary.  

\begin{corollary}\label{coro-homo}
Suppose all the assumptions in Theorem \ref{mainconv} are satisfied. If $\omega(u)$ contains a spatially homogeneous function, then it is the only element of $\omega(u)$.
\end{corollary}


\begin{remark}\label{remark-convergence}{\rm
\begin{itemize}
\item[(a)]
It should be pointed out that Theorem \ref{mainconv} is proved under very general assumptions on $f$ (only \eqref{0state}, \eqref{tperiod} and \eqref{assume-lipschitz}). As already mentioned above, for the autonomous equation \eqref{homo-equation} studied in \cite{dm}, the case (ii) of Theorem \ref{mainconv} does not occur. We suspect that the same is true for the time-periodic problem \eqref{E}, but at the moment, we are only able to exclude the case (ii) under an additional mild non-degeneracy condition on $f$, as stated in Theorem \ref{theorem-non-degenerate} below or for particular types of nonlinearities such as the combustion type, as shown in Proposition \ref{combustion-theorem} below.


\item[(b)] The assumption \eqref{assume-lipschitz} is assumed just for technical reasons and could possibly be weakened. This assumption is only used to show some estimates of the solutions of \eqref{ODE} with respect to initial values (see Lemma \ref{derivative-estimates} and also Step 3 of the proof of Proposition \ref{nonexistence-result}).
\end{itemize}}
\end{remark}

Our second theorem is concerned with the convergence result under a mild non-degenerate assumption on nonnegative $T$-periodic solutions of \eqref{ODE}. Let us introduce a few more notations here. For each $a\geq 0$, we let $h(t;a)$ denote the solution of
\begin{equation}\label{odeinitial}
h_t=f(t,h) \,\,\hbox{ for } t>0;\quad h(0)=a. 
\end{equation}
Then, depending on the sign of $h(T;a)-a$, the behaviors of $h(t;a)$ are classified as follows, as one easily sees from the comparison principle:
\begin{itemize}
\item $h(t;a)$ is {\bf $T$-periodic}: $h(t+T;a)= h(t;a)$ for $t\geq 0$;
\item $h(t;a)$ is {\bf $T$-monotone increasing}: $h(t+T;a)> h(t;a)$ for $t\geq 0$;
\item $h(t;a)$ is {\bf $T$-monotone decreasing}: $h(t+T;a)< h(t;a)$ for $t\geq 0$.
\end{itemize}
We say $h(t;a)$ is {\bf $T$-monotone nondecreasing} (resp. {\bf $T$-monotone nonincreasing}) if it is either $T$-periodic or $T$-monotone increasing (resp. $T$-monotone decreasing).
Let $\mathcal{X}_{per}$ denote the set of all nonnegative $T$-periodic solutions of \eqref{ODE}, and let
$\mathcal{Y}_{per}$ denote
\begin{equation*}
\left.\baa{ll} \mathcal{Y}_{per}: = \Big\{p\in \mathcal{X}_{per}:\,  \hbox{there is }  \epsilon>0 \hbox{ such that }   \!\!\!\! &\hbox{for every } a\in (p(0), p(0)+\epsilon),  \vspace{3pt}\\
\!\!& h(t;a)  \hbox{ is $T$-monotone nondecreasing}  \Big\}.\eaa\right.
\end{equation*}

We assume that:
\begin{itemize}
{\it 
\item[{\bf (H)}]  
Each element $p\in \mathcal{X}_{per}$ either belongs to $\mathcal{Y}_{per}$ or is linearly stable.  }
\end{itemize}

In order to clarify what the assumption (H) signifies, let us first recall some basic notions on stability. An element $p\in \mathcal{X}_{per}$ is said to be {\bf stable from above} (resp. {\bf below}) with respect to \eqref{odeinitial} if it is stable under nonnegative (resp. nonpositive) perturbations of the initial data around $a=p(0)$. Otherwise $p$ is called {\bf unstable from above} (resp. {\bf below}).  It is easily checked that: if $p\in \mathcal{X}_{per}$ is stable from above, then either there exists a sequence of $T$-periodic solutions of \eqref{odeinitial} converging to $p$ from above or there exists $\epsilon>0$ such that $h(t+mT;a)$ is $T$-monotone decreasing and converges to $p$ as $m\to\infty$ for every $a\in (p(0), p(0)+\epsilon)$; $p$ is unstable from above if and only if there exists $\epsilon>0$ such that $h(t+mT;a)$ is $T$-monotone increasing for every $a\in (p(0), p(0)+\epsilon)$. Moreover, we say that $p\in \mathcal{X}_{per}$ is {\bf linearly stable} (resp. {\bf linearly unstable}) if 
$$\int_{0}^{T} \partial_u f(t,p(t))dt < 0  \,\, (\hbox{resp.} >0).  $$ 
It is also easily seen that if $p$ is linearly stable, then there exists $\epsilon>0$ such that $h(t+mT;a)$ is $T$-monotone decreasing for every $a\in (p(0), p(0)+\epsilon)$, and it is $T$-monotone increasing for every $a\in (p(0)-\epsilon, p(0))$.

Note that the assumption (H) does not entirely rule out degenerate nonnegative $T$-periodic solutions of \eqref{ODE}, or even those degenerate periodic solutions that are stable from above. Indeed, there are cases in which $\mathcal{Y}_{per}$ possess an element that is stable from above but not linearly stable. 

\begin{theorem}\label{theorem-non-degenerate}
Let $u$ and $f$ be as in Theorem \ref{mainconv}. If (H) is satisfied, then only the case (i) of Theorem \ref{mainconv} occurs.
\end{theorem}

To explain the role of the assumption (H) in the above theorem, let us give an outline of the proof here. First we show that if $p\in \mathcal{X}_{per}$ is the base of a symmetrically decreasing $T$-periodic solution of \eqref{ground}, then $p$ cannot belong to $\mathcal{Y}_{per}$; see Proposition \ref{nonexistence-result}. The condition (H) is used in the second step, which ensures the uniqueness (up to shift in $x$) of symmetrically decreasing $T$-periodic solutions based at such a $p$. Theorem \ref{theorem-non-degenerate} then follows immediately from Theorem \ref{mainconv}.  

\vskip 5pt
As mentioned above, in the autonomous case $f=f(u)$, the convergence of any bounded solution to an equilibrium was obtained in \cite{dm}. It was also proved in \cite{dm} that if the equilibrium is a positive zero of $f$, say $\gamma$, then none of the following conditions can hold for some $\epsilon>0$:
\begin{equation*}
 f(s)<0 \,\hbox{ for } \,s\in (\gamma-\epsilon,\gamma); 
\qquad  f(s)\leq 0 \,\hbox{ for } \,s \in (\gamma-\epsilon,\gamma+\epsilon).
\end{equation*}
In the following theorem, we show that similar result still holds in the time-periodic case.

\begin{theorem}\label{precise-omega-limit}
Let $u$ and $f$ be as in Theorem \ref{mainconv} and $p\in \mathcal{X}_{per}$. Then $p\notin \tilde{\omega}(u)$
if one of the following conditions holds for some $\epsilon>0$:
\begin{itemize}
\item[{\rm (i)}] $h(t;a)$ is $T$-monotone decreasing for $a\in (p(0)-\epsilon,p(0))$;
\item[{\rm (ii)}] $h(t;a)$ is $T$-monotone nonincreasing for $a\in (p(0)-\epsilon,p(0)+\epsilon)$,
\end{itemize} 
where $h(t;a)$ is the solution of \eqref{odeinitial} with initial value $a$. 
\end{theorem} 

The above theorem in particular implies that if $p\in \mathcal{X}_{per}$ is unstable from below with respect to \eqref{odeinitial}  or it is an intermediate element of a family of $T$-periodic solutions of \eqref{ODE} that forms a continuum,  then it can never belong to $\tilde{\omega}(u)$. 

It should be pointed out that the proof of Theorem \ref{precise-omega-limit} is completely different from that of the autonomous case used in \cite{dm}. To overcome considerable difficulties arising from time dependency $f(t,u)$, we develop new techniques based on change of variables.    

\vskip 5pt

In what follows, we apply the above theorems to problem \eqref{E} with two specific classes of $f$---bistable nonlinearity and combustion nonlinearity, and specify more precisely which $T$-periodic solutions of \eqref{ground} can possibly be selected as the limit.

\vskip 5pt

{\bf (Bistable nonlinearity)} {\it The solution $0$ of \eqref{ODE} is stable from above with respect to \eqref{odeinitial}, and there exist two positive $T$-periodic solutions $p_1$ and $q_1$ of \eqref{ODE} satisfying $0<q_1<p_1$, and $p_1$ is stable from below with respect to \eqref{odeinitial}. Furthermore, 
there are no other positive $T$-periodic solutions of \eqref{ODE}.}

\vskip 5pt

{\bf (Combustion nonlinearity)} {\it There exist a family of $T$-periodic solutions $\{q_{\lambda}\}_{\lambda\in[0,1]}$ of \eqref{ODE} such that $\{q_{\lambda}(0)\}_{\lambda\in[0,1]}$ forms a continuum, and another $T$-periodic solution $p_1$ satisfying $0=q_0<q_1<p_1$ and $p_1$ is stable from below with respect to \eqref{odeinitial}. Furthermore, there are no other positive $T$-periodic solutions of \eqref{ODE}.}

\vskip 5pt

Such two types of nonlinearities appear in modeling various phenomena in applications including mathematical ecology, populations dynamics and combustion; see e.g., \cite{aw1,aw2,fm,k,shen,xin} and references therein. A typical example of these nonlinearities involving time variations is of form
$f(t,u)=b(t)g(u)$, where $b(t)$ is positive and $T$-periodic, and $g$ is a autonomous bistable or combustion nonlinearity. Note that our bistable and combustion nonlinearities cover a much wider variety of nonlinearities which are allowed to change signs with respect to $t$. 

\vskip 5pt

In the bistable case, we have the following convergence result.

\begin{prop}\label{bistable-theorem}
Assume that $f$ is a bistable nonlinearity satisfying \eqref{0state}, \eqref{tperiod}, \eqref{assume-lipschitz} and
\begin{equation}\label{linearly-stable}
\int_{0}^{T} \partial_u f(t,0)dt < 0.  
\end{equation}
Let $u(x,t)$ be a bounded solution of \eqref{E} with nonnegative and compactly supported initial function $u_0$. 
Then we have 
$$
\tilde{\omega}(u)= \big\{0\big\},\,\,\big\{U\big\} \,\hbox{ or }\, \big\{ p_1\big\},   
 $$
where $U(x,t)$ is a symmetrically decreasing $T$-periodic  solution of \eqref{ground} based at $0$.
\end{prop}

The assumption \eqref{linearly-stable} is only used to guarantee the uniqueness of the solution $U(x,t)$ (up to shift in $x$), while its existence is not a priori assumed. Proposition \ref{bistable-theorem} points out that if it exists, then it can possibly be the only element of $\tilde{\omega}(u)$.  

We remark that the above convergence result is not a consequence of the sharp transition result of bistable equations proved in \cite{po1}. The approach in \cite{po1} makes substantial use of the condition \eqref{deri-negative} and that $f(t,u)$ can be bounded from above and below by autonomous bistable nonlinearities, while these assumptions are not needed in our theorem.

\vskip 5pt

The convergence result for combustion equations is stated as follows. 
\begin{prop}\label{combustion-theorem}
Assume that $f$ is a combustion nonlinearity satisfying \eqref{0state}, \eqref{tperiod} and \eqref{assume-lipschitz}. Let $u(x,t)$ be a bounded solution of \eqref{E} with nonnegative and compactly supported initial function $u_0$.
Then we have 
$$
\tilde{\omega}(u)= \big\{0\big\},\,\,\big\{q_1\big\} \,\hbox{ or }\, \big\{ p_1\big\}.
 $$
\end{prop}

\begin{remark}\label{sharp-transition}
{\rm Given the above two propositions, it is also very interesting to investigate whether there is a sharp transition between extinction (i.e., the solution $u(x,t)$ converges to $0$) and propagation (i.e., the solution $u(x,t)$ converges to $p_1$) when the initial data are varied. As remarked above, the sharp transition has been observed in autonomous bistable and combustion equations (\cite{dm,z1}), and also in nonautonomous bistable equations under the condition that $f(t,u)$ can be bounded from above and below by autonomous bistable nonlinearities (\cite{po1}). We believe that, in the time-periodic case, such a condition can be removed if a priori existence of the symmetrically decreasing periodic solution $U(x,t)$ is assumed. Indeed, by using Proposition \ref{bistable-theorem} above and \cite[Theorem 5.1]{po1}, one can show the local instability of  $U(x,t)$, which immediately implies the sharp transition. As for the nonautonomous combustion equations, little has been known about the sharp transition result. We leave this problem for further study.}
\end{remark}

\noindent{\bf Outline of the paper.} In Section 2, we present some preliminaries and show some basic properties of the $\tilde{\omega}$-limit set of bounded solutions of \eqref{E}. Section 3 is concerned with some properties of symmetrically decreasing $T$-periodic solutions of problem \eqref{ground}. First we give a sufficient condition for the existence of such solutions on the entire real line $\R$. Then we show the existence on a finite interval with Dirichlet boundary condition when the base is unstable from above. The techniques developed in this section are key tools in showing our main theorems. In section 4, we complete the proof of Theorems \ref{mainconv} and \ref{theorem-non-degenerate}. Section 5 is devoted to the proof of Theorems \ref{precise-omega-limit} and Propositions \ref{bistable-theorem} and \ref{combustion-theorem}.


\SE{Preliminaries}
In this section, we collect some basic properties which will be needed later. Throughout this section, we assume the nonlinearity $f$ satisfies \eqref{0state} and \eqref{tperiod}.

\subsection{Zero-number properties}
In this subsection, we recall some properties of zero-number arguments. Let $\mathcal{Z}(w)$ denote the number of sign changes of a continuous function $w(x)$ defined on $\R$, namely, the supremum over all $k\in\N$ such that there exist real numbers $x_1<x_2<\cdots<x_{k+1}$ with 
$$ w(x_i)\cdot w(x_{i+1})<0 \,\hbox{ for all } i=1,2,\ldots,k.$$ 
We set $\mathcal{Z}(w)=-1$ if $w\equiv 0$. Clearly, if $w$ is a smooth function having only simple zeros on $\R$, then $\mathcal{Z}(w)$ coincides with the number of zeros of $w$. We also use the notation $\mathcal{Z}_I(w)$ to denote the number of sign changes of $w$ on a given interval $I$.

\begin{lemma}\label{zero1}
Let $w\not\equiv 0$ be a solution of the equation 
\begin{equation}\label{zeroeq}
w_t=w_{xx}+c(x,t)w \quad \hbox{for}\ \ t\in(t_1,t_2),\ x\in\R,
\end{equation} 
where the coefficient function $c$ is bounded. Then the following statements hold:
\begin{itemize}
\item[{\rm (i)}]
For each $t\in (t_1,t_2)$, all zeros of $w(\cdot,t)$ are isolated;
\item[{\rm (ii)}]
$t\mapsto\mathcal{Z}(w(\cdot,t))$ is a nonincreasing function with values in $\N\cup\{0\}\cup\{\infty\}$; 
\item[{\rm (iii)}]
If $w(x^*,t^*)=w_x(x^*,t^*)=0$ for some $t^*\in (t_1,t_2)$, $x^*\in\R$, then
$$\mathcal{Z}(w(\cdot,t))> \mathcal{Z}(w(\cdot,s)) \quad \hbox{for all}\ \ t\in(t_1,t^*),\ s\in(t^*,t_2) $$
whenever $\mathcal{Z}(w(\cdot,s))<\infty$.
\end{itemize}
Furthermore, the same assertion holds for $\mathcal{Z}_I(w(\cdot,t))$ for any interval $I\subseteq \R$, provided that either $w$ never vanishes on the boundary of $I$ or $w\equiv 0$ on the boundary of $I$.
\end{lemma}

The proof of this lemma is referred to \cite{an} when $I$ is a finite interval, and it can be easily extended to the infinite interval case; see the remarks in \cite{dm,dgm}. As an easy application of Lemma \ref{zero1}, we have the following two lemmas.

\begin{lemma}\label{zero2}
Let $w\in C^{2,1} (\R\times (t_1,t_2))$ be a solution of \eqref{zeroeq}. Suppose that there exists $x_0\in\R$ such that $w(x_0,t)=w_x(x_0,t)=0$ for every $t\in (t_1,t_2)$. Then $w\equiv 0$.
\end{lemma}

\begin{lemma}\label{zeroequiv}
Let $w\in C^{2,1} (\R\times (t_1,t_2))$ be a solution of \eqref{zeroeq}. Suppose that there exists $t_0\in (t_1,t_2)$ such that $w(\cdot,t_0)\equiv 0$. Then $w\equiv 0$.
\end{lemma}

One can check that $\mathcal{Z}$ is semi-continuous with respect to pointwise convergence, that is, the pointwise convergence $w_n(x)\to w(x)$ implies 
\begin{equation}\label{semi-continuous}
w\equiv 0 \quad \hbox{or} \quad \mathcal{Z}(w)\leq \liminf_{n\to\infty}\mathcal{Z}(w_n).
\end{equation}

We also recall the following property which will be used frequently later. The proof can be found in \cite{dm}. 

\begin{lemma}\label{zero3}
Let $w_n(x,t)$ be a sequence of functions converging to $w(x,t)$ in $C^1(I\times (t_1,t_2))$, where $I$ is an open finite interval in $\R$. Assume that 
for each $t\in(t_1,t_2)$ and $n\in\N$, the function $x\mapsto w_n(x,t)$ has only simple zeros in $I$, and that $w(x,t)$ satisfies an equation of the form \eqref{zeroeq} on $ I\times (t_1,t_2)$. Then for every $t\in(t_1,t_2)$, either $w\equiv 0$ on $I$ or $w(x,t)$ has only simple zeros on $I$. 
\end{lemma}


\subsection{Properties of bounded solution of \eqref{E} at finite $t$}
Let $u_0\in L^{\infty}(\R)$ be a nonnegative and compact support function. We define its support ${\rm spt}(u_0)$ as the smallest closed set $A\subset \R$ such that $u_0=0$ a.e. in $\R\setminus A$. Throughout this paper, we use the notation $[{\rm spt}(u_0)]$ to denote the convex hull of ${\rm spt}(u_0)$, and put 
\begin{equation}\label{support}
[{\rm spt}(u_0)]:=[L_1,L_2]
\end{equation}
for some $L_1<L_2$.

\begin{lemma}\label{monotone}
Let $u(x,t)$ be the bounded solution of \eqref{E} with initial function $u_0$. Then 
\begin{equation*}
u_x>0\,\hbox{ for } \, x<L_1,\,t>0;\quad u_x<0\,\hbox{ for } \, x>L_2,\,t>0. 
\end{equation*}
Consequently, $\|u(\cdot,t)\|_{L^{\infty}(\R)}=\|u(\cdot,t)\|_{L^{\infty}([L_1,L_2])}$ for all $t\geq 0$. 
\end{lemma}
\begin{proof}
This lemma is proved by a simple reflection argument (see e.g., \cite[Lemma 2.1]{dm}) and we omit the details.
\end{proof}

\begin{lemma}\label{largex}
Let $u(x,t)$ be the bounded solution of \eqref{E} with initial function $u_0$. Then 
\begin{equation}\label{eqlargex}
\lim_{|x|\to\infty}u(x,t)=0\, \hbox{ uniformly in } t\in [0,\tau] \hbox{ for each } \tau>0.   
\end{equation}
\end{lemma}
\begin{proof}
The proof follows directly from a simple comparison argument and the fact that \eqref{eqlargex} holds for the solution of heat equation (see e.g., \cite[Lemma 2.2]{dm}). 
\end{proof}

\begin{lemma}\label{zeorperiod}
Let $u(x,t)$ be the bounded solution of \eqref{E} with initial function $u_0$ and $v(x,t)$ be a positive bounded entire solution of \eqref{ground}. Then 
\begin{equation*}
\mathcal{Z}(u(\cdot,t+T)-u(\cdot,t))<\infty \,\hbox{ for } t>0,
\end{equation*}
and 
\begin{equation*}
\mathcal{Z}(u(\cdot,t)-v(\cdot,t))<\infty \,\hbox{ for } t>0.
\end{equation*}
Moreover, both of them are nonincreasing in $t>0$. 
\end{lemma}

\begin{proof}
We only prove the conclusions for $u(\cdot,t+T)-u(\cdot,t)$, as the proof of the other one is similar.
Since $f$ satisfies \eqref{tperiod}, $u(x,t+T)$ is also a solution of \eqref{equation}. Then $u(\cdot,t+T)-u(\cdot,t)$ solves a linear equation of form \eqref{zeroeq}, where $c(x,t)$ is bounded because of the uniform boundedness of $u$ and $f\in C^1([0,\infty))$ with respect to $u$ (the local Lipschitz continuity of $f(\cdot,u)$ is sufficient).

Let $L>\max\{|L_1|,|L_2|\}$ with $L_1$, $L_2$ given in \eqref{support}. We can choose $\delta>0$ sufficiently small such that 
$$0<u(\pm L, t)< u(\pm L, t+T) \,\hbox{ for } 0<t\leq \delta, $$
since $u(x,t)>0$ for $x\in\R,\,t>0$ by the strong parabolic maximum principle.  
Note that $u(x,0)<u(x,T)$ for all $|x|\geq L$. By comparison principle, we have
$$u(x,t)<  u(x,t+T) \quad \hbox{ for }  0<t\leq \delta,\, |x|\geq  L. $$
It then follows from Lemma \ref{zero1} that
$$ \mathcal{Z}(u(\cdot,t+T)-u(\cdot,t))<\infty \,\hbox{ for } 0<t\leq \delta,$$
and that $\mathcal{Z}(u(\cdot,t+T)-u(\cdot,t))$ is nonincreasing in $t>0$. The lemma is thus proved. 
\end{proof}


\subsection{Basic properties of $\omega$-limit solutions}

We recall that $\tilde{\omega}(u)$ is the set of all $\omega$-limit solutions of $u(x,t)$. Namely, for each $v\in \tilde{\omega}(u)$ and each $t\geq 0$,  we have 
\begin{equation}\label{limitsolu}
u(x,t+k_j T) \to v(x,t)  \ \ \ \hbox{as}\ \ k_j\to\infty \,\,\hbox{ in } L^{\infty}_{loc}(\R)
\end{equation}
for some sequence of positive integers $k_j$. In this subsection, we prove some basic properties of the set $\tilde{\omega}(u)$ by applying zero-number arguments. 
\vskip 5pt

We first show the following asymptotic symmetric property. 

\begin{lemma}\label{symmprop}
Let $L_1<L_2$ be given in \eqref{support}. Then there exists $x_0\in [L_1,L_2]$ such that any $v \in \tilde{\omega}(u)$ is symmetrically nonincreasing with respect to $x_0$, that is, for each $t\in\R$, 
\begin{equation}\label{symm}
v(x,t)=v(2x_0-x,t)\,\hbox{ for } x\in\R, 
\end{equation}
and 
\begin{equation}\label{sdec}
 v_x(x_0,t)=0,\quad  v_x(x_0,t)\geq 0 \,\hbox{ for } x<x_0,\quad v_x(x_0,t)\leq 0 \,\hbox{ for } x>x_0.
\end{equation}
\end{lemma}
\begin{proof}
It follows from exactly the same arguments as those used in \cite[Lemmas 2.8, 3.2]{dm} that, there exists $x_0\in [L_1,L_2]$ such that each $v \in \tilde{\omega}(u)$ satisfies 
\begin{equation}\label{peak}
v_x(x_0,t)=0,\quad  v_x(x_0,t)\geq 0 \,\hbox{ for } x<x_0,\,t\in\R.
\end{equation}
Indeed, $x_0$ is chosen as the limit position of the leftmost local maximum of the function $t\mapsto u(x,t)$, i.e., $$x_0:=\inf\big\{x\in\R^1:\, \lim_{t\to\infty} sgn(u_x(x,t))=-1 \big\},$$ 
where $sgn(w):=1,\,-1,\,0$ depending on whether $w>0$, $w<0$ or $w=0$. 

We now show that $v(x,t)$ is symmetric with respect to $x_0$ by considering the function
$$\varphi(x,t):=v(x,t)-v(2x_0-x,t)  \quad\hbox{ for } x\in\R,\,t\in\R.$$
Since both $v(x,t)$ and $v(2x_0-x,t)$ are bounded solutions of the same equation \eqref{ground}, $\varphi$ satisfies a linear parabolic equation of form \eqref{zeroeq} with bounded $c(x,t)$. Moreover, since $\varphi(x_0,t)=\varphi_x(x_0,t)=0$ for each $t\in\R$, it then follows from Lemma \ref{zero2} that $\varphi \equiv 0$, and hence \eqref{symm} holds. This together with \eqref{peak} immediately implies \eqref{sdec}. The proof of Lemma \ref{symmprop} is complete. 
\end{proof}

In what follows, let $x_0$ denote this point of symmetry described as in the above lemma. 

\begin{lemma}\label{lem:periodic-tangential}
Let $v_1$ be a $T$-periodic $\omega$-limit solutions of \eqref{E} and $v_2$ be a nonnegative entire solution of \eqref{ground} which is $T$-periodic in $t$ and symmetric with respect to $x_0$. Then 
\begin{equation}\label{two-periodic-solution}
\hbox{either}\quad v_1\equiv v_2 \quad\hbox{or}\quad  \mathcal{Z}(v_1(\cdot,t)-v_2(\cdot,t))<\infty \,\hbox{ for each } t\in\R.
\end{equation}
Moreover,  if $ v_1(x_0,t_0)=v_2(x_0,t_0)$ for some $t_0\in\R$, then $v_1\equiv v_2$.
\end{lemma}

\begin{proof}
Since $v_1$ and $v_2$ are nonnegative, it is clear that all the conclusions hold if $v_1\equiv 0$ or $v_2\equiv 0$. Then we only need to consider the case where $v_1$ and $v_2$ are both positive. By Lemma \ref{zeorperiod}, we have
$$\mathcal{Z}(u(\cdot,t)-v_2(\cdot,t))<\infty \,\hbox{ for each } t>0.$$
In particular, there holds
$$\mathcal{Z}(u(\cdot,t+k_jT)-v_2(\cdot,t))<\infty \,\hbox{ for each } t>0,\,k_j\in\N,$$
where $\{k_j\}$ is the sequence of positive integers such that \eqref{limitsolu} holds with $v$ replaced by $v_1$. 
By the semi-continuity property \eqref{semi-continuous},  we immediately obtain \eqref{two-periodic-solution}. 

Next, we prove that if  $ v_1(x_0,t_0)=v_2(x_0,t_0)$ for some $t_0\in\R$, then $v_1\equiv v_2$.
Suppose the contrary that $v_1\not\equiv v_2$. It then follows from Lemma \ref{zero1} that $\mathcal{Z}(v_1(\cdot,t)-v_2(\cdot,t))$ is nonincreasing in $t\in\R$, and hence it is a constant for all large $t$. This implies 
all the zeros of $v_1(\cdot,t)-v_2(\cdot,t)$ are simple for all large $t$.  On the other hand, since $v_1$ and $v_2$ are $T$-periodic and symmetric with respect to $x_0$, we have
$$ v_1(x_0,t_0+kT)=v_2(x_0,t_0+kT) \quad\hbox{and}\quad \partial_x v_1(x_0,t_0+kT)=\partial_x v_2(x_0,t_0+kT)=0,$$
for every $k\in\Z$. Thus, $x_0$ is a degenerate zero of $v_1(\cdot,t_0+kT)-v_2(\cdot,t_0+kT)$, which is a contradiction. Therefore, there must hold $v_1\equiv v_2$. The proof of this lemma is complete.    
\end{proof}

\begin{lemma}\label{conorsym}
Let $v$ be any $\omega$-limit solution of \eqref{E}. Then either of the following holds:
\begin{itemize}
\item[{\rm (a)}]   
$v(x,t)$ is a symmetrically decreasing with respect to $x_0$; 
\item[{\rm (b)}] 
 $v=v(t)$ is spatially homogeneous.
\end{itemize}
\end{lemma}

\begin{proof}
Since $v$ is symmetrically nonincreasing with respect to $x_0$ by Lemma \ref{symmprop}, it suffices to show that if there exist $s_0\in\R$ and $y_0 \neq x_0$ such that $v_x(y_0,s_0)=0$, then $v$ is spatially homogeneous. Without loss of generality, we assume that $y_0>x_0$.  

We first prove $v(\cdot,s_0)$ is constant. Define
$$\tilde{\varphi}(x,t):= u(x,t)- u(x+2x_0-2y_0,t) \quad\hbox{for }\,t\geq 0,\,x\in\R.$$ 
It is easily seen that $\tilde{\varphi}(x,t)$ satisfies a linear parabolic equation of form \eqref{zeroeq} with bounded coefficient. 
 
We claim that $\tilde{\varphi}(\cdot,t)$ has only finite simple zeros for all large $t$. 
Since $f(t,0)\equiv 0$ and $u$ is bounded, there exists $M\geq 0$ such that $-Mu\leq f(t,u)\leq Mu$ for all $t\geq 0$. Therefore, 
$$\exp (-Mt) \bar{u}(x,t)\leq u(x,t)\leq \exp (Mt)\bar{u}(x,t) \,\, \hbox{ for } x\in\R,\, t\geq 0, $$
where $\bar{u}$ is the solution of $u_t=u_{xx}$ with initial function $u_0$.  As a consequence, for $x\in\R$, $t>0$, we have
\begin{equation*}
\begin{split}
\frac{u(x,t)}{u(x+2x_0-2y_0,t)}
\,&\, \geq \exp (-2Mt) \frac{\bar{u}(x,t)}{ \bar{u}(x+2x_0-2y_0,t)}\\
\,&\,=\exp (-2Mt) \frac{\displaystyle\int_{L_1}^{L_2} \hbox{exp}\Big(-\frac{(x-y)^2}{4t}\Big) u_0(y)dy}{\displaystyle\int_{L_1}^{L_2} \hbox{exp}\Big(-\frac{\big(x-y+2(x_0-y_0)\big)^2}{4t}\Big) u_0(y)dy}, 
\end{split}
\end{equation*}
where the last quantity tends to $\infty$ as $x\to-\infty$, due to $y_0>x_0$. And hence, 
\begin{equation*}
\frac{u(x,t)}{u(x+2x_0-2y_0,t)}\to \infty \quad \hbox{ as } x\to-\infty.
\end{equation*}
In a similar way, we obtain 
\begin{equation*}
\frac{u(x,t)}{u(x+2x_0-2y_0,t)}\to 0 \quad \hbox{ as } x\to\infty.
\end{equation*}
Then for each given $t_0>0$, there exists 
$L>0$ large enough such that 
$$\tilde{\varphi}(x,t_0) >0 \,\hbox{ for }  x\leq -L,\quad   \tilde{\varphi}(x,t_0) <0 \,\hbox{ for }  x\geq L. $$
This together with the fact that the zeros of $x\mapsto \tilde{\varphi}(x,t_0)$ do not accumulate in $\R$ implies 
$\mathcal{Z}(\tilde{\varphi}(\cdot,t_0))<\infty$. By Lemma \ref{zero1}, $\mathcal{Z}(\tilde{\varphi}(\cdot,t))$ is constant for all large $t$, and hence, 
$\tilde{\varphi}(\cdot,t)$ has only finite simple zeros in $\R$ for all large $t$. The proof of our claim is finished. 

Next, let $\{k_j\}_{j\in\N}$ be the sequence of positive integers such that \eqref{limitsolu} holds. 
By standard parabolic estimates, 
$$\tilde{\varphi}(x,t+k_jT)\,\to\, v(x,t)-v(x+2(x_0-y_0),t) \, \hbox{ as } j\to \infty \hbox{ in }\, C^1(I\times (t_1,t_2)), $$
where $I\subset \R$ is any finite interval containing $y_0$ and $(t_1,t_2)\subset \R$ is any finite interval containing $s_0$. In view of this and Lemma \ref{zero3}, we obtain 
\begin{equation*}
\hbox{ either }\,v(x,s_0)-v(x+2(x_0-y_0),s_0)\equiv 0 \,\hbox{ on } I
\end{equation*}
\begin{equation*}
\hbox{ or }\,v(x,s_0)-v(x+2(x_0-y_0),s_0) \,\hbox{ has only simple zeros on } I.
\end{equation*}
The latter is impossible, since by Lemma \ref{symmprop}, $x=y_0$ is a degenerate zero of $v(x,s_0)-v(x+2(x_0-y_0),s_0)$. Therefore, the former case happens. Furthermore, by the arbitrariness of finite interval $I$, we have $v(\cdot,s_0)\equiv  v(\cdot+2(x_0-y_0), s_0)$ on $\R$. 
Since $v(x,s_0)$ is symmetrically nonincreasing with respect to $x_0$,  we obtain $v(\cdot,s_0)$ is constant. 

Lastly, we show that $v=v(t)$ is spatially homogeneous by considering the function $v_x(x,t)$ over $(x,t)\in \R ^2$. Clearly, $v_x(x,t)$ satisfies
the following equation
\begin{equation*}
w_t=w_{xx}+\partial_v f(t,v)w \,\hbox{ for } x\in \R,\, t\in \R.
\end{equation*}
Note that $\partial_v f(t,v(x,t))$ is bounded because of the boundedness of $v$. Since $v_x(\cdot, s_0) \equiv 0$, it then follows from Lemma \ref{zeroequiv} that $v_x(\cdot,t)\equiv 0$ for all $t\in\R$, and hence $v=v(t)$ is spatially homogeneous.
The lemma is thus proved. 
\end{proof}

\begin{lemma}\label{omega-sol}
Let $v$ be any $\omega$-limit solution of \eqref{E}. Then either
of the following holds:
\begin{itemize}
\item[\rm (a)] $v(x,t)$ is a $T$-periodic solution of \eqref{ground};
\item[\rm (b)] for each $t\in\R$, $v(x_0,t+mT)$ is strictly monotone  
in $m\in\Z$, and $v(x,t+mT)$ converges to $T$-periodic solutions of \eqref{ground} as 
$m\to\pm\infty$ locally uniformly in $x\in\R$, $t\in\R$.
\end{itemize}
\end{lemma}

\begin{proof}
We assume that $v$ is not a $T$-periodic solution of \eqref{ground}, and show (b) occurs.

We first prove that for each $t\in\R$, $v(x_0,t+mT)$ is strictly monotone in $m\in \Z$. 
It suffices to show that either $v(x_0,t+T)>v(x_0,t)$ for each $t\in\R$ or $v(x_0,t+T)<v(x_0,t)$ for each $t\in\R$. 
Assume by contraction that neither of them holds. Then by the continuity of $v$, there exists $t_0\in\R$ such that 
\begin{equation}\label{degezero1}
v(x_0,t_0+T)=v(x_0,t_0).
\end{equation} 
Let $\{k_j\}_{j\in\N}$ be the sequence of positive integers such that \eqref{limitsolu} holds, and for each $j\in\N$, write 
$$\psi_j(x,t):=u(x,t+k_jT+T)-u(x,t+k_jT) \,\hbox{ for } t>0,\,x\in\R.   $$
By Lemma \ref{zeorperiod}, for each $j\in\N$ and $t>0$, $\mathcal{Z}(\psi_j(\cdot,t))<\infty$, 
and it is nonincreasing in $t>0$. Let $(t_1,t_2)\in (0,\infty)$ be a finite interval containing $t_0$. Then  $\mathcal{Z}(\psi_j(\cdot,t))$ is constant for all $t\in (t_1,t_2)$ when $k$ is sufficiently large. Consequently, $\psi_j(\cdot,t)$ has only simple zeros for all large $k$. Moreover, by standard parabolic estimates, we have
$$ \psi_j(x,t)\to v(x,t+T)- v(x,t)\,\hbox{ as } k_j\to\infty\, \hbox{ in } C^1_{loc}(\R\times (t_1,t_2)). $$
It then follows from Lemma \ref{zero3} that, either $v(\cdot,t_0+T)- v(\cdot,t_0)\equiv 0 $ on $\R$ or  $v(\cdot,t_0+T)- v(\cdot,t_0) $ has only simple zeros on $\R$. The latter is impossible since $x=x_0$ is a degenerate zero of $v(\cdot,t_0+T)- v(\cdot,t_0)$, due to \eqref{sdec} and \eqref{degezero1}. Therefore, 
$v(\cdot,t_0+T)\equiv v(\cdot,t_0)$ on $\R$. Furthermore, applying Lemma \ref{zeroequiv} to the linear equation satisfied by $v(\cdot,t+T)-v(\cdot,t)$, we obtain for each $t\in\R$,  
$$v(\cdot,t+T)\equiv v(\cdot,t) \,\hbox{ on } \R. $$
This contradicts our assumption that $v(x,t)$ is not $T$-periodic. Thus, $v(x_0,t+mT)$ is strictly monotone in $m\in \Z$. 

Next, we prove that $v(x,t+mT)$ converges to $T$-periodic solutions of \eqref{ground} as $m\to\pm\infty$ locally uniformly.  
We only consider the convergence as $m\to\infty$, as the proof for the other case is exactly the same. 
By standard parabolic estimates, there exists an entire solution $V(x,t)$ of \eqref{ground} such that 
\begin{equation*}
v(x,t+m_jT) \to V(x,t) \,\hbox{ as } m_j\to\infty \,\hbox{ in } C^{1}_{loc} (\R\times \R) 
\end{equation*}
for some sequence $\{m_j\}_{j\in\N}$ of positive integers. Suppose that there exists another subsequence $\{\tilde{m}_j\}_{j\in\N}\subset \N$ such that $v(x,t+\tilde{m}_jT)$ converges to another entire solution $\tilde{V}(x,t)$ of \eqref{ground} as $\tilde{m}_j\to\infty$ in $C^{1}_{loc} (\R\times\R)$. 
Since $v(x_0,t+kT)$ is strictly monotone in $k\in \Z$, it is clear that 
$$V(x_0,t)= \tilde{V}(x_0,t) \,\hbox{ for } t\in\R.$$
Moreover, in view of \eqref{sdec}, we have 
$$V_x(x_0,t)=\tilde{V}_x(x_0,t)=0 \,\hbox{ for } t\in\R.$$  
Then, by applying Lemma \ref{zero2} to the linear equation satisfied by $V(x,t)-\tilde{V}(x,t)$, we obtain $V\equiv \tilde{V}$. Therefore, the whole sequence $v(x,t+mT)$ converges as $m\to\infty$. Clearly, the limit function is $T$-periodic in $t$.  

Therefore, either (b) or (c) happens. The proof of Lemma \ref{omega-sol} is thus complete.
\end{proof}

Let $V_{\pm}(x,t)$ denote the limit functions of the sequence $v(x,t+mT)$ as $m\to\pm\infty$ if Case (b) in Lemma \ref{omega-sol} holds, that is,
\begin{equation}\label{omega-lim-pm}
v(x,t+mT) \to V_{\pm}(x,t) \, \hbox{ as } m \to\pm\infty  \,\hbox{ locally uniformly in } x\in\R,\,t\in\R.
\end{equation}
Since the set $\tilde{\omega}(u)$ is closed in $L^{\infty}_{loc}(\R\times\R)$, 
$V_{\pm}$ are also $\omega$-limit solutions of \eqref{E}. We end this section by showing the existence of another $\omega$-limit solution connecting $V_{\pm}(x,t)$, as stated in the following proposition.

\begin{prop}\label{hereconnect}
Let $v$ be any $\omega$-limit solution and suppose that $v$ is not $T$-periodic in $t$. 
If $v(x_0,t+mT)$ is increasing in $m\in\N$, then there exists another $\omega$-limit solution $w(x,t)$ of problem \eqref{E} such that $w(x_0,t+mT)$ is decreasing in $m\in\N$, and that 
\begin{equation}\label{omega-lim-mp}
w(x,t+mT) \to V_{\mp}(x,t) \, \hbox{ as } m \to\pm\infty  \,\hbox{ locally uniformly in } x\in\R,\,t\in\R,
\end{equation}
where $V_{\pm}(x,t)$ is given in \eqref{omega-lim-pm}. 

Similarly, if $v(x_0,t+mT)$ is decreasing in $m\in\N$, then there exists another $\omega$-limit solution $\tilde{w}(x,t)$ such that $\tilde{w}(x_0,t+mT)$ is increasing in $m\in\N$, and that \eqref{omega-lim-mp} holds with $w$ replaced by $\tilde{w}$. 
\end{prop}

\begin{proof}
We only give the proof for the case where $v(x_0,t+mT)$ is increasing in $m\in\N$, since the analysis for the other case is identical. In this case, it is clear that $V_-(x_0,t)<V_+(x_0,t)$ for $t\in\R$. 
For clarity, we divide the proof into two steps.

{\bf Step 1:} There exists no positive $T$-periodic solution $\Phi(x,t)$ of \eqref{ground} which is symmetric with respect to $x_0$ and satisfies 
\begin{equation}\label{conbeween}
V_-(x_0,t_0)<\Phi(x_0,t_0)<V_+(x_0,t_0)\, \hbox{ for  some } t_0\in\R. 
\end{equation}

Assume by contradiction that there exists such a solution $\Phi(x,t)$. Then by \eqref{omega-lim-pm}, \eqref{conbeween} and the continuity of the functions $\Phi$ and $v$, there exists $\tilde{t}_0\in [0,T)$ and $n_0\in\Z$ such that
$$v(x_0,\tilde{t}_0+n_0T)=\Phi(x_0,\tilde{t}_0).$$ 
Since $v(x,\tilde{t}_0+n_0T)$ and $\Phi(x,\tilde{t}_0)$ are symmetric with respect to $x_0$, this implies that $x=x_0$ is a degenerate zero of the function $v(\cdot,\tilde{t}_0+n_0T)-\Phi(\cdot,\tilde{t}_0)$.

By Lemma \ref{zeorperiod},  $\mathcal{Z}(u(\cdot,t)-\Phi(\cdot,t))<\infty$ and it is nonincreasing in $t>0$.  It then follows from Lemma \ref{zero1} that
$$V_j(x,t):=u(x,t+k_jT+n_0T)-\Phi(x,t) $$
has only finite simple zeros in $\R$ for all large $j\in \N $ and all $t\in [0,T)$, where $\{k_j\}_{j\in\N}$ is the sequence of positive integers such that \eqref{limitsolu} holds. Since 
$$ V_j(x,t) \to v(x,t+n_0T)-\Phi(x,t) \,\hbox{ as } k_j\to\infty \hbox{ in } C_{loc}^1(\R\times \R),$$
and since $x=x_0$ is a degenerate zero of $v(\cdot,\tilde{t}_0+n_0T)-\Phi(\cdot,\tilde{t}_0)$, it follows from Lemma \ref{zero3} that  $v(\cdot,\tilde{t}_0+n_0T)\equiv \Phi(\cdot,\tilde{t}_0)$. Furthermore, by Lemma \ref{zeroequiv}, we have $v(\cdot,t+n_0T)\equiv \Phi(\cdot,t)$ for each $t\in [0,T)$. This implies that $v(\cdot,t)$ is a $T$-periodic solution of \eqref{ground}, which is a contradiction with our assumption. The proof of Step 1 is complete.

{\bf Step 2: }Completion of the proof. 

Choose $a>0$ such that 
\begin{equation*}
V_-(x_0,0)<a<V_+(x_0,0). 
\end{equation*}
Due to \eqref{limitsolu} and \eqref{omega-lim-pm}, we can find a sequence $\{m_i\}_{i\in\N}\subset \Z$ with $m_i\to\infty$ as $i\to\infty$ such that 
\begin{equation}\label{auxi-a}
V_-(x_0,0)<u(x_0,(m_i+1)T)\leq a < u(x_0,m_iT)<V_+(x_0,0) 
\end{equation}
for each $i\in\N$. Then by parabolic estimates, there exists an $\omega$-limit solution $w(x,t)$ of \eqref{E} such that, up to extraction of some subsequence of $\{m_i\}_{i\in\N}$, there holds
\begin{equation*}
u(x,t+m_iT) \to w(x,t)  \ \ \ \hbox{as}\ \ m_i\to\infty \hbox{ locally uniformly in } x\in\R,\,t\in\R.
\end{equation*}
In view of \eqref{auxi-a}, we have
\begin{equation}\label{auxieqT}
V_-(x_0,0)\leq w(x_0,T)\leq a \leq w(x_0,0)\leq V_+(x_0,0) .  
\end{equation}
Clearly, $w\not\equiv V_{\pm}$. We further conclude that 
\begin{equation}\label{intermed}
V_-(x_0,t)< w(x_0,t)< V_+(x_0,t)\,\hbox{ for each } t\in\R.  
\end{equation}
Otherwise, there would exist some $s_0\in\R$ such that either $V_-(x_0,s_0)= w(x_0,s_0)$ or $V_+(x_0,s_0)= w(x_0,s_0)$. Assume without loss of generality that the former occurs. Then $x=x_0$ is a degenerate zero of $V_-(\cdot,s_0)- w(\cdot,s_0)$. By similar arguments as those used in Step 1, one would derive $V_-\equiv w$, which is impossible.

It then follows from Step 1 that $w(t,x)$ can not be $T$-periodic. This together with Lemma \ref{omega-sol} and \eqref{auxieqT} implies that for each $t\in\R$, $w(x_0,t+mT)$ is decreasing in $m\in\N$, and $w(x,t+mT)$ converges to $T$-periodic solutions of \eqref{ground} as $m\to\pm\infty$ locally uniformly in $(x,t)\in\R\times\R$. Denote the limit functions by $W_{\pm}$, respectively.  It is easily seen from \eqref{intermed} that
\begin{equation*}
V_-(x_0,t) \leq W_+(x_0,t)<W_-(x_0,t)\leq V_+(x_0,t)\quad \hbox{ for } \,t\in\R.
\end{equation*}
Now, to obtain \eqref{omega-lim-mp}, it suffices to show that $W_{\pm}\equiv V_{\mp}$. Since they are all $T$-periodic $\omega$-limit solutions of \eqref{E}, in view of Lemma \ref{lem:periodic-tangential}, we only need to prove $W_{\pm}(x_0,t)= V_{\mp}(x_0,t)$ for some $t\in\R$. 
Suppose the contrary that it does not hold. Then either
\begin{equation*}
V_-(x_0,t)<W_+(x_0,t)<  V_+(x_0,t)\, \hbox{ for all } t\in\R,
\end{equation*}
or 
\begin{equation*}
V_-(x_0,t)<W_-(x_0,t) <  V_+(x_0,t)\, \hbox{ for all } t\in\R
\end{equation*}
holds. Both are contradictions with the conclusion of Step 1. Thus, \eqref{omega-lim-mp} is obtained  and the proof of Proposition \ref{hereconnect} is complete. 
\end{proof}


\SE{Symmetrically decreasing periodic solutions of \eqref{ground}} 

In this section, we show some properties of the symmetrically decreasing $T$-periodic solutions of \eqref{ground}. In Section 3.1, we give a sufficient condition for the existence of such solutions on the entire real line $\R$. Section 3.2 is concerned with the existence of a finite-interval counterpart problem with Dirichlet boundary condition. These properties as well as the techniques developed in showing them are key ingredients to prove our main theorems later. Besides, they may be also of interest in their own. We emphasis that, throughout this section, the nonlinearity $f$ is only assumed to satisfy  \eqref{tperiod} and \eqref{assume-lipschitz}.


\subsection{A sufficient condition for the existence result}

We recall that $\mathcal{X}_{per}$ and $\mathcal{Y}_{per}$ are the subsets of $T$-periodic solutions of \eqref{ODE} defined in Section 1. In this subsection, we prove the following proposition.

\vskip 5pt

\begin{prop}\label{nonexistence-result}
Let $f$ satisfy \eqref{tperiod} and \eqref{assume-lipschitz} and $p\in \mathcal{X}_{per}$. If problem \eqref{ground} admits a symmetrically decreasing $T$-periodic solution $U(x,t)$ based at $p$, then $p\notin \mathcal{Y}_{per}$. 
\end{prop}

To show this proposition, we will construct a lower solution of the symmetrically decreasing periodic solution $U(x,t)$, but it cannot decay to any element of $\mathcal{Y}_{per}$ as $|x|\to\infty$. The proof is based on change of variables and the estimates in the following two lemmas.

We first show some estimates of the derivates of $h(t;a)$ with respect to $a$, where $h(t;a)$ is the solution of 
\eqref{odeinitial} with initial value $a$. Note that if $f(\cdot,u)$ is $C^2$ in $u$, it is well known that $h(t;a)$ is also a $C^2$ function in $a$. In our discussion below, we use $h_a(t;a)$ and $h_{aa}(t;a)$ to denote $\partial h / \partial a (t;a)$ and $\partial^2 h / \partial a^2 (t;a)$, respectively, if they exist.

\begin{lemma}\label{derivative-estimates}
Let $p_{\pm} \in \mathcal{X}_{per}$ satisfying $p_{-}<p_{+}$. Suppose that $f(\cdot,u)$ is a $C^2$ function in $u$. Then we have
\begin{equation}\label{esti-ha}
h_{a}(t;a)\geq M_1\,\,\hbox{ for } t\in [0,T], \, a\in [p_{-}(0),p_{+}(0)],
\end{equation}
and 
\begin{equation}\label{esti-haa}
-M_2\leq \frac{h_{aa}(t;a)}{h_a(t;a)} \leq M_2 \,\,\hbox{ for } t\in [0,T], \, a\in [p_{-}(0),p_{+}(0)],
\end{equation}
where $M_1$ and $M_2$ are two positive constants depending only on $T$ and 
\begin{equation}\label{defi-C1-C2}
C_1:=  \sup_{t\in [0,T],\, u\in[p_{-}(t),p_{+}(t)] } \big| \partial_{u} f(t,u)  \big|, \quad\quad C_2:= \sup_{t\in [0,T],\, u\in[p_{-}(t),p_{+}(t)] } \big| \partial_{uu} f(t,u)  \big|.
\end{equation}
\end{lemma}

\begin{proof}
By the differentiability of the solutions of \eqref{odeinitial} with respect to initial values (see, e.g. \cite{cl}), it is clear that $h_a(t;a)$ and $h_{aa}(t;a)$ are, respectively, the solutions of
\begin{equation}\label{differential-1} 
\frac{d h_a}{dt}=\partial_uf(t,h(t;a))h_a \,\,\hbox{ for } \,0<t\leq T, \quad h_a(0)=1,
\end{equation}
and 
\begin{equation}\label{differential-2} 
\frac{d h_{aa}}{dt}=\partial_{uu}f(t,h(t;a))h_a^2+\partial_uf(t,h(t;a))h_{aa}\,\, \hbox{ for } \,0<t\leq T, \quad h_{aa}(0)=0.
\end{equation}
By a simple comparison argument applied to \eqref{odeinitial} and \eqref{differential-1}, for each $a\in [p_{-}(0),p_{+}(0)]$, we have 
$$p_-(t) \leq h(t;a)\leq p_+(t) \,\,\hbox{ for } t\in [0,T],  $$
and 
$$h_-(t) \leq h_a(t;a)\leq h_+(t) \,\,\hbox{ for } t\in [0,T],  $$
where $h_{\pm}(t)$ are solutions of 
$$ \displaystyle\frac{d h_{\pm}}{dt}=\pm C_1h_{\pm}\,\,  \hbox{ for } 0<t\leq T, \quad h_{\pm}(0)=1. $$
It then easily calculated that 
$$\exp(-C_1t) \leq  h_a(t;a) \leq \exp(C_1t)  \,\,\hbox{ for } t\in [0,T], \, a\in [p_{-}(0),p_{+}(0)]. $$
This in particular gives the estimate \eqref{esti-ha} by setting $M_1:=\exp(-C_1T)$. 

Similarly, applying comparison principle to \eqref{differential-2},  we obtain  
$$\tilde{h}_-(t) \leq h_{aa}(t;a)\leq \tilde{h}_+(t)  \,\,\hbox{ for } t\in [0,T], \, a\in [p_{-}(0),p_{+}(0)]. $$
where $\tilde{h}_{\pm}(t)$ are solutions of 
$$\frac{d \tilde{h}_{\pm}}{dt}=\pm C_1\tilde{h}_{\pm} \pm C_2\exp( 2C_1T) \,\,  \hbox{ for } 0<t\leq T, \quad \tilde{h}_{\pm}(0)=0. $$
Thus, there holds
 $$
\frac{C_2\exp( 2C_1T)}{C_1}\Big(\exp(-C_1t)-1\Big)\leq  h_{aa}(t;a) \leq  \frac{C_2\exp( 2C_1T)}{C_1}\Big(\exp(C_1t)-1\Big)
$$
for $t\in [0,T]$, $a\in [p_{-}(0),p_{+}(0)]$. This immediately yields \eqref{esti-haa} by choosing 
$$M_2:=\max\Big\{\frac{C_2\exp( 2C_1T)}{C_1}\Big(1-\exp(-C_1T)\Big),\frac{C_2\exp( 2C_1T)}{C_1}\Big(\exp(C_1T)-1\Big)   \Big\}.$$
The proof of Lemma \ref{derivative-estimates} is thus complete. 
\end{proof}

\vskip 5pt

In the proof of Proposition \ref{nonexistence-result} below, by a change of variable, we will write the symmetrically decreasing periodic solution $U(x,t)$ as a composition of a solution of \eqref{odeinitial} and a solution of a semilinear parabolic equation. The following lemma will be used in the lower estimates of the solution of the semilinear equation as $t\to\infty$.

\begin{lemma}\label{estimate-trans-heat}
Let $M>0$, $\delta_0>0$ and $l_0\in \R$ be given constants and let $w_0\in L^{\infty}((-\infty,l_0])$ satisfy $w_0(x)=0$ for all large negative $x$. Then we have 
\begin{equation}\label{w-to-delta0}
\lim_{t\to\infty} w(x,t) =\delta_0 \,\,\hbox{ locally uniformly in } x\in (-\infty,l_0],
\end{equation}
where $w(x,t)$ is the solution of 
\begin{equation}\label{subsolu-M}
\left\{\baa{ll}
\smallskip   w_t= w_{xx}-M w_x^2, & \hbox{ for } \,x<l_0,\,t>0, \vspace{3pt}\\ 
\smallskip w(x,0)= w_0(x), & \hbox{ for }\,x\leq l_0, \vspace{3pt}\\
w(l_0, t)=\delta_0, & \hbox{ for }\, t\geq 0.\eaa\right.
\end{equation}
\end{lemma}

\begin{proof}
Taking 
$$\psi(x,t)= \exp\big(-M w(x,t) \big) \,\,\hbox{ for }  x\leq l_0,\,t\geq 0,$$
and substituting it into equation \eqref{subsolu-M}, we obtain that $\psi(x,t)$ is the solution of
\begin{equation*}
\left\{\baa{ll}
\smallskip   \psi_t= \psi_{xx}, & \hbox{ for } \,x<l_0,\,t>0, \vspace{3pt}\\
\smallskip  \psi(x,0)= \exp\big(-Mw_0(x)\big), & \hbox{ for }\,x\leq l_0, \vspace{3pt}\\
\psi(l_0, t)= \exp\big(-M\delta_0\big), & \hbox{ for }\, t\geq 0.\eaa\right.
\end{equation*}
Therefore, to prove \eqref{w-to-delta0}, it amounts to show 
\begin{equation}\label{psi-to-expdelta0}
\lim_{t\to\infty} \psi(x,t) =\exp\big(-M\delta_0\big)  \,\,\hbox{ locally uniformly in } x\in (-\infty,l_0].
 \end{equation}
To do this, we consider the heat equation 
\begin{equation*}
\left\{\baa{ll}
\smallskip   \varphi_t= \varphi_{xx}, & \hbox{ for } \,x\in\R,\,t>0, \vspace{3pt}\\
\varphi(x,0)= \varphi_0(x), & \hbox{ for }\, x\in\R,\eaa\right.
\end{equation*}
where $\varphi_0 \in L^{\infty}(\R)$ is an odd function with respect to $l_0$ given by 
\begin{equation*}
\varphi_0(x):=\left\{\baa{ll}
\smallskip   \exp\big(-Mw_0(x)\big)-\exp\big(-M\delta_0\big), & \hbox{ for } \,x< l_0, \vspace{3pt}\\
\smallskip   0, & \hbox{ for } \,x= l_0, \vspace{3pt}\\
\exp\big(-M\delta_0\big) -  \exp\big(-Mw_0(2l_0-x)\big), & \hbox{ for }\, \,x>l_0.\eaa\right.
\end{equation*}
One easily checks that $\varphi(x,t)=-\varphi(2l_0-x,t)$ for $x\in\R$, $t\geq 0$, and hence,
$\varphi(l_0,t)=0$ for all $t\geq 0$. This implies that
\begin{equation}\label{psi-varphi-expo}
\psi(x,t)= \varphi(x,t)+ \exp\big(-M\delta_0\big) \,\,\hbox{ for } x\leq l_0, \,t\geq 0. 
\end{equation}
On the other hand, since $w_0(x)=0$ for all sufficiently large negative $x$, there exists $l_1>0$ sufficiently large such that 
$$\varphi_0(x)= \exp\big(-M\delta_0\big) - 1\,\,\hbox{ for } x\geq l_1,\qquad \varphi_0(x) = 1- \exp\big(-M\delta_0\big)\,\,\hbox{ for } x\leq -l_1. $$
It follows that 
\begin{equation*}
\begin{split}
\varphi(x,t)
\medskip\,&\, = \displaystyle \int_{-\infty}^{\infty} \varphi_0(y) K(t,x-y)  dy \vspace{3pt}\\
\medskip  \,&\,\displaystyle   = \Big(1-\exp\big(-M\delta_0\big) \Big) \Big[ \int_{-\infty}^{-l_1}  K(t,x-y)  - \int_{l_1}^{\infty} K(t,x-y)dy \Big] \vspace{3pt}\\ 
\medskip  \,&\displaystyle \qquad\qquad+\int_{-l_1}^{l_1} \varphi_0(y) K(t,x-y)  dy,
\end{split}
\end{equation*}
where $K$ is the fundamental solution of the heat equation, i.e. 
$$ K(t,x):=\frac{1}{\sqrt{4\pi t}} \exp\Big( \frac{-x^2}{4t} \Big) .  $$
It is then straightforward to compute that
$$\varphi(x,t)\to  0 \hbox{ as } t\to\infty \hbox{ locally uniformly in } x\in\R.$$
Combining this with \eqref{psi-varphi-expo}, we immediately obtain \eqref{psi-to-expdelta0}, and hence \eqref{w-to-delta0} is proved. 
\end{proof}

\vskip 5pt

With the above preparations, we are ready to give the proof of Proposition \ref{nonexistence-result}.

\begin{proof}[Proof of Proposition \ref{nonexistence-result}]
For the sake of convenience, we may assume without loss of generality that $p\equiv 0$. Indeed, if this not satisfied, one can replace the nonlinearity $f(t,u)$ by $f(t,u+p)-p_t$, then $U(x,t)-p(t)$ is a symmetrically decreasing $T$-periodic solution of \eqref{ground} (with the new nonlinearity) based at $0$. Moreover, this change does not affect the behavior of solutions of \eqref{odeinitial} with initial values close to $p(0)$.

We also assume without loss of generality that $x=0$ is the symmetric point of $U(x,t)$, and that there exists an element $\bar{p}\in\mathcal{X}_{per}$ such that 
$$U(x,t)< \bar{p}(t) \quad\hbox{ for } x\in\R,\, t\in\R.$$
Indeed, if this new condition is not satisfied, since $U(x,t)$ is bounded, one can achieve it by easily modifying the values of $f(t,u)$ where $u$ is far away above the range of $U(x,t)$.

Suppose by contrary that $0\in \mathcal{Y}_{per}$. Then, since $\lim_{|x|\to\infty} U(x,t)=0$, we can find $l_0<0$ sufficiently large negative such that the solution $h(t;a)$ of \eqref{odeinitial} is $T$-monotone nondecreasing for each $a\in [0,U(l_0,0)]$. Since $U(x,t)$ is increasing in $x\in (-\infty,l_0]$, for any given $x\in (-\infty,l_0]$, we have
\begin{equation}\label{find-l0}
h(t;a) \hbox{ is $T$-monotone nondecreasing for each } a\in [0,U(x,0)].
\end{equation}
Let $v(x,t) \in C^1((-\infty,l_0)\times [0,T])$ be determined by the relation
\begin{equation}\label{change-variable-psi}
U(x,t)= h(t;v(x,t)) \,\,\hbox{ for }  \, x\in (-\infty,l_0],\,t\in [0,T].
\end{equation}
By simple comparison arguments, we conclude that $v(x,t)$ is increasing in $x\in (-\infty,l_0)$,
\begin{equation}\label{claim-xtoinfty-vto0}
v(x,t) \to 0 \,\,\hbox{ as } x\to -\infty \,\hbox{ uniformly in } t\in [0,T],
\end{equation}
and 
\begin{equation}\label{bound-of-v}
0<v(x,t)<\bar{p}(t)\,\,\hbox{ for }  x\in (-\infty,l_0],\,t\in [0,T]. 
\end{equation}

For clarity, we divide the following analysis into 4 steps.

{\bf Step 1: }In this step, we show that for any given $x\in (-\infty,l_0]$, $h(t;a)$ is $T$-monotone nondecreasing for each $a\in [0,v(T,x)]$.

To show this, in view of \eqref{find-l0}, it suffices to prove that 
\begin{equation}\label{U0-geq-vT}
U(x,0)\geq v(x,T) \,\, \hbox{ for all } x\in (-\infty, l_0].
\end{equation}
Indeed, since $U(x,t)$ is $T$-periodic in $t$, taking $t=0$ and $t=T$ in \eqref{change-variable-psi}, we immediately obtain 
\begin{equation}\label{recurrence}
v(x,0)=h(T,v(x,T)) \,\, \hbox{ for all } x\in (-\infty, l_0].  
\end{equation}
Assume by contradiction that \eqref{U0-geq-vT} does not hold. Then there exists $y_0\in (-\infty, l_0]$ such that $U(y_0,0)< v(y_0,T)$. It then follows from \eqref{find-l0} at $x=y_0$ that
$$v(y_0,0)=h(T,v(y_0,T))> h(T,U(y_0,0)) \geq  U(y_0,0),  $$ 
which is an apparent contradiction with the fact that $v(y_0,0)=U(y_0,0)$. Thus, the conclusion of this step is proved. 

{\bf Step 2:} We show a lower estimates of $v(x,t)$ under the additional assumption that $f(\cdot,u)$ is a $C^2$ function in $u$.

Since $f(\cdot,u)$ is $C^2$ in $u$, it is easily seen that $h(t;a)$ is a $C^2$ function with respect to $a$. Moreover, it is also straightforward to check that 
$v\in C^{2,1}((-\infty,l_0)\times [0,T])$, and that $U(x,t)$ is a solution of \eqref{ground} over $(-\infty,l_0)\times [0,T]$ if and only if 
$v(x,t)$ is a solution of 
\begin{equation*}
h_a(t;v)v_t=h_a(t;v)v_{xx}+ h_{aa}(t;v)v_x^2\,\,\hbox{ for }\, x\in (-\infty,l_0),\,t\in [0,T], 
\end{equation*}
or equivalently, 
\begin{equation}\label{equation-variable-v}
v_t=v_{xx}+ \frac{h_{aa}(t;v)}{h_a(t;v)}v_x^2\,\,\hbox{ for }\, x\in (-\infty,l_0),\,t\in [0,T],
\end{equation}
since $h_a(t;v)$ is positive over $(-\infty,l_0)\times [0,T]$ due to the estimate \eqref{esti-ha}. 

By \eqref{bound-of-v} and Lemma \ref{derivative-estimates}, we have 
\begin{equation}\label{esitmate-haa/ha}
\frac{h_{aa}(t;v)}{h_a(t;v)} \geq -M \,\, \hbox{ for } \, x\in (-\infty,l_0),\,t\in [0,T], 
\end{equation}
where $M$ is a positive constant depending only on 
$$T,\quad  \sup_{t\in [0,T],\, u\in[0,\bar{p}(t)] } \big| \partial_{u} f(t,u)  \big| \quad\hbox{and}\quad\sup_{t\in [0,T],\, u\in[0,\bar{p}(t)]  } \big| \partial_{uu} f(t,u)  \big|.$$
Note that $\inf_{t\in [0,T]} v(x,t)>0$ for every $x\in (-\infty,l_0]$. We can choose a nonnegative and nonzero function $w_0\in C((-\infty, l_0])$ satisfying 
\begin{equation}\label{initial-cond-w0}
w_0(x) \leq  v(0,x) \,\hbox{ for } x\leq l_0, \qquad  w_0(x)=0 \,\hbox{ for sufficiently large negative } x.  
\end{equation}
Let $w(x,t)$ be the solution of \eqref{subsolu-M} with initial function $w_0$, $M$ given by \eqref{esitmate-haa/ha}, and 
$$\delta_0:=\frac{1}{2} \inf_{t\in [0,T]} v(l_0,t)>0.$$
Applying the comparison principle to the equations \eqref{equation-variable-v} and \eqref{subsolu-M}, we obtain
\begin{equation*}
v(x,t) \geq w(x,t)  \,\, \hbox{ for } \, x\in (-\infty,l_0],\,t\in [0,T].
\end{equation*}
In particular, we have $v(x,T) \geq w(x,T)$ for $x\in (-\infty,l_0]$. It then follows from the conclusion of Step 1 and \eqref{recurrence} that 
\begin{equation}\label{using-monote-nondec}
v(x,0)=h(T,v(x,T))\geq  h(T,w(x,T)) \geq w(x,T) \,\, \hbox{ for } \, x\in (-\infty,l_0]. 
\end{equation}
By the comparison principle again, we have 
\begin{equation*}
v(x,t) \geq w(x,T+t) \,\, \hbox{ for } \, x\in (-\infty,l_0],\,t\in [0,T].
\end{equation*}
Then, a simple induction argument immediately gives
\begin{equation}\label{lower-extimate-v}
v(x,t) \geq w(x,nT+t)  \,\, \hbox{ for } \, x\in (-\infty,l_0],\,t\in [0,T],\,n\in\N.
\end{equation}

{\bf Step 3:} We show the estimate \eqref{lower-extimate-v} in the case where $f$ satisfies \eqref{assume-lipschitz}

In this case, we can choose a sequence of locally H{\"o}lder continuous functions $(f_k)_{k\in\N}$ in $\R\times[0,\infty)$ such that for each $k\in\N$, $f_k(\cdot,u)$ is a $C^2$ function in $u\geq 0$, that
$$f_k(t,u)  \nearrow f(t,u)  \,\,\hbox{ as } k\to\infty \,\hbox{ locally uniformly in } u\geq 0,\, t\in \R, $$  
and that 
\begin{equation}\label{approximate}
\partial_u f_k(t,u),\,\, \partial_{uu} f_k(t,u)  \hbox{ are uniformly bounded in }  u\in [0,\bar{p}(t)],\,t\in [0,T],\,k\in\N,  
\end{equation}
where $\bar{p}(t)$ is given in \eqref{bound-of-v}. For each $k\in\N$, let $v^k(x,t)$ be the function determined by the relation 
$$ U(x,t)= h^k(t;v^k(x,t)) \,\,\hbox{ for }  \, x\in (-\infty,l_0],\,t\in [0,T],$$
where for any $a>0$, $ h^k(t;a)$ is the solution of the ODE 
$$h^k_t=f_k(t,h^k) \,\,\hbox{ for } t>0;\quad h^k(0)=a. $$
One then easily checks that $v^k(x,t)$ is nonincreaasing in $k\in\N$, 
and that for each $k\in\N$, $v^k\in C^{2,1}((-\infty,l_0)\times [0,T])$ satisfies
$$v^k_t=v^k_{xx}+ \frac{h^k_{aa}(t;v^k)}{h^k_a(t;v^k)}(v^k_x)^2+ \frac{f(t,h^k)-f_k(t,h^k)}{h^k_a(t;v^k)}
\,\,\hbox{ for }\, x\in (-\infty,l_0],\,t\in [0,T].$$
It is also clear that
\begin{equation*}
v^k(x,t) \to v(x,t)\,\, \hbox{ as } k\to\infty \,\hbox{ pointwisely in } \, x\in (-\infty,l_0],\,t\in [0,T], 
\end{equation*}
where $v(x,t)$ is the function determined by \eqref{change-variable-psi}.

Moreover, because of \eqref{approximate}, by the proof of Lemma \ref{derivative-estimates}, we see that $h^k_a(t;v^k)$ is uniformly positive and $h^k_{aa}(t;v^k)/h^k_a(t;v^k)$ is uniformly bounded in $x\in (-\infty,l_0],\,t\in [0,T]$ with the bounds independent of $k\in\N$. By a slight abuse of notation, we still use $-M$ to denote the lower bound of $h^k_{aa}(t;v^k)/h^k_a(t;v^k)$. It then follows that, for each $k\in\N$, $v^k(x,t)$ satisfies 
$$v^k_t\geq v^k_{xx}-M(v^k_x)^2 \,\,\hbox{ for }\, x\in (-\infty,l_0],\,t\in [0,T].$$
Let $w(x,t)$ be the solution of \eqref{subsolu-M} with such $M$ and $w_0$, $\delta_0$ given as in Step 2.
By the comparison principle, we have
\begin{equation*}
v^k(x,t) \geq w(x,t)  \,\, \hbox{ for } \, x\in (-\infty,l_0],\,t\in [0,T],\,k\in\N.
\end{equation*}
Sending the limit to $k\to\infty$, we obtain  
\begin{equation*}
v(x,t) \geq w(x,t)  \,\, \hbox{ for } \, x\in (-\infty,l_0],\,t\in [0,T].
\end{equation*}
Then the same reasoning as in showing \eqref{using-monote-nondec} gives $v(x,0) \geq w(x,T) $ for $x\in (-\infty,l_0]$, whence $v^k(x,0) \geq w(x,T)$ for each $k\in\N$. By the comparison principle again, we have
\begin{equation*}
v^k(x,t) \geq w(x,t+T)  \,\, \hbox{ for } \, x\in (-\infty,l_0],\,t\in [0,T],\,k\in\N.
\end{equation*}
An induction argument implies that, for each $k\in\N$, 
\begin{equation*}
v^k(x,t) \geq w(x,nT+t)  \,\, \hbox{ for } \, x\in (-\infty,l_0],\,t\in [0,T],\,n\in\N.
\end{equation*}
Therefore, passing the limit to $k\to\infty$ in the above inequality immediately gives \eqref{lower-extimate-v}.

{\bf Step 4:} Completion of the proof. 

By Lemma \ref{estimate-trans-heat}, we have
\begin{equation*}
\lim_{t\to\infty} w(x,t) =\delta_0 \,\,\hbox{ locally uniformly in } x\in (-\infty,l_0].
\end{equation*}
This together with \eqref{lower-extimate-v} implies 
$$ v(x,t) \geq \delta_0  \,\,\hbox{ for } x\in (-\infty, l_0), \,t\in [0,T],  $$
which appears to be a contradiction with \eqref{claim-xtoinfty-vto0}. Therefore, $0\notin \mathcal{Y}_{per}$, and  the proof of Proposition \ref{nonexistence-result} is complete. 
\end{proof}

\vskip 5pt

As an easy application of Proposition \ref{nonexistence-result}, we have the following corollary.

\begin{corollary}\label{coro-intermediate}
Suppose that problem \eqref{ground} admits a symmetrically decreasing $T$-periodic solution $U(x,t)$ based at some $p\in \mathcal{X}_{per}$.  Then $p$ is stable from above with respect to \eqref{odeinitial}, and
there exists another $q\in \mathcal{X}_{per}$ such that
\begin{equation}\label{pleqqlequ}
p(t)<q(t)< \max_{x\in\R}U(x,t)\,\,\hbox{ for } t\in\R. 
\end{equation}
\end{corollary}

\begin{proof}
The stability of $p$ with respect to \eqref{odeinitial} follows directly from Proposition \ref{nonexistence-result}. 
To prove \eqref{pleqqlequ}, we assume by contradiction that there is no such an element $q\in \mathcal{X}_{per}$. Let $\bar{p}$ be the smallest element in $\mathcal{X}_{per}$ such that $U(x,t)<\bar{p}(t)$ for $x\in\R$, $t\in\R$ (see the beginning of the proof of Proposition \ref{nonexistence-result} for the existence of such $\bar{p}$). It is clear that for each $a\in \big(p(0), \bar{p}(0)\big)$,
\begin{equation*}
h(t+mT;a)  \to  p(t) \,\hbox{ as } m\to\infty \hbox{ locally uniformly in }  t\in\R,
\end{equation*}
which is a contradiction with the existence of $U(x,t)$. 
\end{proof}

\vskip 15pt

\subsection{A Dirichlet boundary problem}

Let  $p_{+}>p_{-}$ be two elements of $\mathcal{X}_{per}$ satisfying that for each $a\in \big(p_-(0), p_+(0)\big)$, 
\begin{equation}\label{unstablec}
h(t+mT;a) \nearrow  p_{+}(t) \,\hbox{ as } m\to\infty \hbox{ locally uniformly in }  t\in\R,
\end{equation}
where $h(t;a)$ is the solution of \eqref{odeinitial}. In other words, $p_+$ is stable from below and $p_-$ is unstable from above with respect to \eqref{odeinitial}. This subsection is concerned with the existence of  solutions of the following periodic-parabolic problem with Dirichlet boundary condition 
\begin{equation}\label{steady-state-equation}
\left\{\baa{ll}
\smallskip  \varphi_t=\varphi_{xx}+f(t,\varphi), & \hbox{ for } \,-R<x<R,\,t\in\R, \vspace{3pt}\\
\smallskip \varphi(x, t+T)=\varphi(x,T), & \hbox{ for }\,-R\leq x\leq R,\,t\in\R, \vspace{3pt}\\
\smallskip p_-(t) <\varphi(x,t)<p_{+}(t), & \hbox{ for }\,-R< x< R,\,t\in\R, \vspace{3pt}\\
\varphi(\pm R, t)=p_-(t), & \hbox{ for }\, t\in\R ,\eaa\right.
\end{equation}
where $R$ is a positive constant. 
\vskip 5pt

The main result of this subsection is stated as follows.

\begin{prop}\label{existence-bound-interval}
Let $f$ satisfy \eqref{tperiod} and \eqref{assume-lipschitz} and assume \eqref{unstablec} holds. Then for all sufficiently large $R$, \eqref{steady-state-equation} admits a symmetrically decreasing solution.
\end{prop}

\vskip 5pt

The existence of solution of \eqref{steady-state-equation} (when $R$ is large) as well as the uniqueness is well known if $p_-$ is linearly unstable (see e.g., \cite{He}). Proposition \ref{existence-bound-interval} gives the existence in the general case that  $p_-$ is unstable from above. This general existence result will be a key ingredient in proving Theorem \ref{precise-omega-limit} in Section 5.

The proof of Proposition \ref{existence-bound-interval} relies on a perturbation argument on problem \eqref{odeinitial}. 
More precisely, for each $a\in [p_-(0),p_+(0)]$ and $\epsilon\geq 0$, let $h(t;a;\epsilon)$ denote the solution of 
\begin{equation}\label{odeperturb}
h_t=f(t,h)-\epsilon \,\hbox{ for } \,t>0;  \quad  h(0)=a. 
\end{equation}
Then the following lemma holds. 

\vskip 5pt

\begin{lemma}\label{perturb}
Let $f$ and $p_{\pm}$ be given as in Proposition \ref{existence-bound-interval}.  
Then there exists a smooth function 
$$g: \big[p_-(0), p_+(0)\big] \mapsto [0,\infty)$$ 
satisfying
$$g\big( p_-(0) \big)=g\big( p_+(0) \big)=0,$$
and 
$$g(a)>0,\quad  h\big(T;a;g(a)\big)>a \quad \hbox{ for each } \,a \in  \big(p_-(0), p_+(0)\big).$$
\end{lemma}

\begin{proof}
It is clear that $h(t;a)$ is $T$-monotone increasing for each $a\in \big(p_-(0), p_+(0)\big)$, that is,
$h(T;a)>a$. Due to the continuous dependence of the solution of \eqref{odeperturb} with respect to the perturbation $\epsilon$, for each $a\in \big(p_-(0), p_+(0)\big)$, there exists 
some $\epsilon_a>0$ such that 
$$h(T;a;\epsilon)>a\,\hbox{ for each }  \epsilon \in [0, \epsilon_a).$$
Denote the supremum of such $\epsilon_a$ by $\epsilon_a^*$. Clearly, $\epsilon_a^*>0$, and 
\begin{equation}\label{supremum}
h(T;a;\epsilon_a^*)=a \,\hbox{ for each } a\in \big(p_-(0), p_+(0)\big).
\end{equation}

Set $\epsilon_a^*=0$ for $a=p_{\pm}(0)$. We now prove that $\epsilon_a^*$ is continuous in $a\in [p_-(0), p_+(0)]$. It suffices to show that for any sequence $\{a_i\}_{i\in\N} \subset [p_-(0), p_+(0)]$ satisfying $a_i\to a_0$ as $i\to\infty$ for some $a_0\in [p_-(0), p_+(0)]$, there holds $\epsilon_{a_i}^* \to \epsilon_{a_0}^*$ as $i\to\infty$. The case where $a_0\in (p_-(0), p_+(0))$ follows directly from \eqref{supremum}, the definition of $\epsilon_a^*$ as well as the continuous dependence of $h(t;a;\epsilon)$ with respect to $\epsilon$ and $a$. Thus, we only need to vertify the case where $a_0=p_{\pm}(0)$. Suppose the contrary that $ \epsilon_{a_i}^*$ converges to some constant $\epsilon_0>0$ as $a_i \to p_{-}(0)$ (the case where $a_i \to p_{+}(0)$ can be treated similarly). It then follows from \eqref{supremum} that $h(T;p_{-}(0);\epsilon_0)=p_{-}(0)$, and hence 
$$h(T;p_{-}(0);\epsilon)>p_{-}(0)  \,\hbox{ for each } \epsilon \in [0,\epsilon_0), $$
which is a apparent contradiction with the fact that $p_{-}(t)$ is $T$-periodic. Therefore, $a\mapsto\epsilon_a^*$ is continuous in $a\in [p_-(0), p_+(0)]$.

It then easily seen that any smooth function $g: \big[p_-(0), p_+(0)\big] \mapsto [0,\infty)$ satisfying
\begin{equation}\label{con-gamma}
g(a)=0 \,\hbox{ for }  a=p_{\pm}(0)\quad\hbox{and}\quad   0<g(a)\leq \epsilon_a^* \,\hbox{ for }  a \in \big(p_{-}(0), p_{+}(0)\big) 
\end{equation}
is the desired function. The proof of this lemma is complete. 
\end{proof}

\vskip 5pt

For each $a\in [p_{-}(0), p_{+}(0)]$, let $H(t;a)$  denote the solution of 
\begin{equation*}
\frac{d H}{dt}=f(t,H)-g(a) \,\hbox{ for } \,0<t\leq T,  \quad  H(0)=a,
\end{equation*}
where $g$ is the perturbation function obtained in Lemma \ref{perturb}. Clearly, if $f(\cdot,u)$ is of class $C^2$ in $u$, then $H(t;a)$ is a $C^2$ function in $a$. Let $C_1$ and $C_2$ be the constants given in  \eqref{defi-C1-C2}. Then we have the following parallel result to Lemma \ref{derivative-estimates}.

\begin{lemma}\label{derivative-estimates-pertur}
Let $g$ be the function obtained in Lemma \ref{perturb}. Suppose that $f(\cdot,u)$ is a $C^2$ function in $u$. Then there exist positive constants $\tilde{M}_1$, $\tilde{M}_2$ and $\tilde{M}_3$ depending only on 
$$T,\,\, C_1,\,\, C_2,\,\,  \big\| g'\big\|_{L^{\infty}([p_{-}(0), p_{+}(0)])} \quad\hbox{and}\quad \big\| g''\big\|_{L^{\infty}([p_{-}(0), p_{+}(0)])}   $$
such that
\begin{equation*}
\tilde{M}_1\leq H_{a}(t;a)\leq \tilde{M}_2\,\,\hbox{ for } t\in [0,T], \, a\in [p_{-}(0),p_{+}(0)],
\end{equation*}
and 
\begin{equation*}
-\tilde{M}_3\leq \frac{H_{aa}(t;a)}{H_a(t;a)} \leq \tilde{M}_3 \,\,\hbox{ for } t\in [0,T], \, a\in [p_{-}(0),p_{+}(0)], 
\end{equation*}
where $H_a(t;a)$ and $H_{aa}(t;a)$ stands for $\partial H / \partial a (t;a)$ and $\partial^2 H / \partial a^2 (t;a)$, respectively. 
\end{lemma}

\begin{proof}
We only show the uniformly positivity of $H_{a}(t;a)$, as the proof of the other estimates is almost identical to that of Lemma \ref{derivative-estimates}. It is easily checked that $H_a(t;a)$ is the solution of 
\begin{equation*}
  \displaystyle\frac{d H_a}{dt}=\partial_uf(t,H)H_a -g'(a) \hbox{ for } \,0<t\leq T;\qquad H_a(0)=1.
\end{equation*}
By a simple comparison argument, we obtain 
$$H_-(t) \leq H_a(t;a) \,\hbox{ for } t\in [0,T],  $$
where $H_{-}(t)$ is the solution of 
$$ \displaystyle\frac{d H_{-}}{dt}=- C_1 H_{-}- C_3\,  \hbox{ for } 0<t\leq T; \quad H_{-}(0)=1. $$
Here $C_3:= \big\| g'\big\|_{L^{\infty}([p_{-}(0), p_{+}(0)])}$. It is clear that
$$H_{-}(t)= -\frac{C_3}{C_1}\big(1-\exp (- C_1t)\big)+\exp (- C_1t)\,\hbox{ for } \,0\leq t\leq T. $$
To show $H_{-}(t)$ is positive in $[0,T]$, we may require the function $g$ obtained in Lemma \ref{perturb} to satisfy
\begin{equation*}
C_3 \leq \frac{C_1}{2(\exp (C_1T)-1)}.
\end{equation*}
Indeed, the existence of such a $g$ follows directly from the fact that any smooth function satisfying \eqref{con-gamma} is the desired function in Lemma \ref{perturb}. 
It then follows that 
$$H_{-}(t)\geq \frac{\exp(- C_1t)-1}{2\big(\exp (C_1T)-1\big)}+\exp(- C_1t)\geq  \frac{1}{2}\exp(- C_1T)\,\hbox{ for } 0\leq t\leq T.$$
And hence, we have 
$$ H_a(t;a)\geq \frac{1}{2}\exp(- C_1T)  \,\hbox{ for } t\in [0,T]. $$
The proof of Lemma \ref{derivative-estimates-pertur} is complete.
\end{proof}

\vskip 5pt
The proof of Proposition \ref{existence-bound-interval} also relies on the following observation.

\vskip 5pt
\begin{lemma}\label{phase-plane}
Let $g$ be the function obtained in Lemma \ref{perturb}, and $c_1$, $c_2$ be two given positive constants.  
Then for all large $R>0$,  there  exists a symmetrically decreasing function $\phi \in C^2([-R,R])$ satisfying 
\begin{equation}\label{variable-change-1}
\left\{\baa{l}
\smallskip  \phi''-c_1(\phi')^2=-c_2g(\phi) \,\, \hbox{ for }-R<x<R,\quad  \phi(\pm R)= p_-(0),   \vspace{3pt}\\
 p_{-}(0)<\phi(x)<  p_{+}(0)\,\,\hbox{ for }-R<x<R.\eaa\right.
\end{equation}
\end{lemma}

\begin{proof}
Suppose that there exist a constant $R>0$ and a symmetrically decreasing function $\phi\in C^2([-R,R])$ satisfying \eqref{variable-change-1}. 
Set
$$\tilde{\phi}(x):= \exp \big(-c_1p_-(0)\big)- \exp \big(-c_1\phi(x)\big) \,\hbox{ for } x\in [-R,R],$$
and define
$$\tilde{g}(q): =c_1c_2 \big( \exp\big(-c_1p_-(0)\big)-q \big) g \big(-c_1^{-1} \ln \big( \exp (-c_1p_-(0)\big)- q) \big) $$
for  $q\in [0,\bar{q}]$, where 
$$\bar{q}:=\exp \big(-c_1p_-(0)\big)- \exp \big(-c_1p_+(0)\big). $$
Clearly, we have 
\begin{equation}\label{monostable}
\tilde{g}(0)=\tilde{g}(\bar{q})=0\quad\hbox{and}\quad \tilde{g}(q)>0 \,\hbox{ for } q\in (0,\bar{q}). 
\end{equation}
It is also easily calculated that 
$$\tilde{\phi}'=c_1\exp (-c_1\phi) \phi'  \quad \hbox{and}\quad \tilde{\phi}''=c_1\exp (-c_1\phi)\big(\phi''-c_1(\phi')^2\big). $$
Thus, $\tilde{\phi}\in C^2([-R,R])$ satisfies 
\begin{equation}\label{variable-change-2}
\left\{\baa{l}
\smallskip   \tilde{\phi}''+\tilde{g}(\tilde{\phi})=0 \,\, \hbox{ for }-R<x<R,\quad \tilde{\phi}(\pm R)=0,   \vspace{3pt}\\
0<\tilde{\phi}(x)< \bar{q}\,\,\hbox{ for }-R<x<R.\eaa\right.
\end{equation}
Therefore, to show this lemma, it amounts to prove that for all large $R>0$, there exists a symmetrically decreasing function $\tilde{\phi}\in C^2([-R,R])$ satisfying \eqref{variable-change-2}. In view of  \eqref{monostable}, $\tilde{g}$ is a monostable nonlinearity.  The existence of the solution $\tilde{\phi}$ can be obtained by elementary phase plane analysis (see e.g., \cite[Lemma 4.1]{dlou}). The proof of Lemma \ref{phase-plane} is thus complete.
\end{proof}

\vskip 5pt

We are able to complete the proof of Proposition \ref{existence-bound-interval}.

\begin{proof}[Proof of Proposition \ref{existence-bound-interval}]
We only give the proof in the case where $f(\cdot,u)$ is a $C^2$ function in $u$, 
as for the general case where $f$ satisfies \eqref{assume-lipschitz}, the following analysis is still valid by a standard approximation argument. 

Let $$ c_1:= \tilde{M}_3 \quad\hbox{and}\quad c_2:= \tilde{M}_2^{-1},$$
where $\tilde{M}_2$ and $\tilde{M}_3$ are the positive constants obtained in Lemma \ref{derivative-estimates-pertur}. It follows from Lemma \ref{phase-plane} that for all large $R>0$, there exists a symmetrically decreasing function $\phi\in C^2([-R,R])$ satisfying \eqref{variable-change-1} with such $c_1$ and $c_2$.
This immediately implies
\begin{equation}\label{phi-inequality}
\phi''+ \frac{H_{aa}(t;a)}{H_a(t;a)} (\phi')^2\geq  -\frac{1}{H_a(t;a)} g(\phi) 
\end{equation}
for all $x\in (-R,R)$, $t\in [0,T]$ and $a\in [p_{-}(0), p_{+}(0)]$. 

Next, for each $x\in [-R,R]$, let $\Phi(x,t)$ denote the unique solution of the following ODE
\begin{equation*}
\Phi_t=f(t,\Phi)-g\big(\phi(x)\big) \,\hbox{ for } \,0<t\leq T;  \quad  \Phi(x,0)=\phi(x).
\end{equation*}
It is clear that
$$\Phi(x,t)= H(t; \phi(x)) \,\,\hbox{ for }  0\leq t\leq T,\, -R\leq x\leq R. $$
Since $f(\cdot,u)$ and $g(u)$ are $C^2$ functions in $u$, by the chain rule, it is straightforward to check that $\Phi(x,t)$ is a $C^2$ function in $x$, and there holds
\begin{equation*}
\partial_{xx}\Phi(x,t) = H_a (t;\phi(x)) \phi''(x)+  H_{aa}(t;\phi(x)) (\phi')^2(x) \,\,\hbox{ for } x\in (-R,R),\,t\in[0,T].  
\end{equation*}
Moreover, since $p_{-}(0)<\phi(x)<  p_{+}(0)$  for $-R<x<R$, it follows from \eqref{phi-inequality} that 
\begin{equation*}
\partial_{xx}\Phi(x,t)\geq - g(\phi(x)) \,\hbox{ for }  x\in (-R,R),\,t\in [0,T].
\end{equation*}
Thus, $\Phi(x,t)$ satisfies
\begin{equation*}
\left\{\baa{ll}
\smallskip \Phi_t \leq \partial_{xx}\Phi+f(t,\Phi), & \hbox{ for } \,-R<x<R,\,0<t\leq T, \vspace{3pt}\\
\smallskip \Phi(x,0)=\phi(x), & \hbox{ for }\, \,-R\leq x\leq R, \vspace{3pt}\\
\smallskip \Phi(\pm R, t)=p_-(t), & \hbox{ for }\,0\leq t\leq T.\eaa\right.
\end{equation*}
It then follows from the comparison principle that 
$$w(x,t;\phi) \geq \Phi(x,t) \,\hbox{ for } x\in [-R,R],\,t\in[0,T],$$
where $w(x,t;\phi)$ is the solution of the following problem
\begin{equation}\label{dirichlet}
\left\{\baa{ll}
\smallskip  w_t=w_{xx}+f(t,w), & \hbox{ for } \,-R<x<R,\,t>0, \vspace{3pt}\\
\smallskip w(x,0)=\phi(x), & \hbox{ for }\, \,-R\leq x\leq R, \vspace{3pt}\\ 
w(\pm R, t)=p_-(t), & \hbox{ for }\, t\geq 0.\eaa\right.
\end{equation}
Furthermore, since $\Phi(x;T)\geq \phi(x)$ for each $x\in [-R,R]$ by Lemma \ref{perturb}, it follows that 
$$w(x,T;\phi) \geq  \phi(x) \,\hbox{ for } x\in [-R,R].$$
Applying the comparison principle to \eqref{dirichlet}, we obtain that $w(x,t+mT;\phi)$ is nondecreasing in $m\in\N$. By standard parabolic estimates, 
\begin{equation*}
w(x,t+mT;\phi)\to \varphi(x,t)\,\,\hbox{ as } m\to\infty \hbox{ in } C_{loc}^{2,1}(\R\times\R),   
\end{equation*}
where $\varphi(x,t)$ is a solution of \eqref{steady-state-equation}. Clearly, $\varphi(x,t)$ is symmetrically decreasing. The proof of Proposition \ref{existence-bound-interval} is thus complete.  
\end{proof}

\begin{remark}\label{re-steady-state}
{\rm  It is easily seen from the proof of Proposition \ref{existence-bound-interval} that similar existence result holds if $p_{-}$ is unstable from below. More precisely, if $p_{-}>p_{+}$ and the condition \eqref{unstablec} is replaced by that for each $a\in (p_+(0),p_-(0))$, 
\begin{equation*}
h(t+mT;a) \searrow  p_{+}(t) \,\hbox{ as } m\to\infty \hbox{ locally uniformly in }  t\in\R,
\end{equation*}
then for all sufficiently large $R$, the following problem 
\begin{equation*}
\left\{\baa{ll}
\smallskip  \varphi_t=\varphi_{xx}+f(t,\varphi), & \hbox{ for } \,-R<x<R,\,t\in\R, \vspace{3pt}\\
\smallskip \varphi(x, t+T)=\varphi(x,T), & \hbox{ for }\,-R\leq x\leq R,\,t\in\R, \vspace{3pt}\\
\smallskip p_+(t) <\varphi(x,t)<p_{-}(t), & \hbox{ for }\,-R< x< R,\,t\in\R, \vspace{3pt}\\
\varphi(\pm R, t)=p_-(t), & \hbox{ for }\, t\in\R ,\eaa\right.
\end{equation*}
admits a symmetrically increasing solution. }
\end{remark}

\vskip 15pt
\SE{Proof of Theorems \ref{mainconv} and \ref{theorem-non-degenerate}} 
In this section, we give the proof of our main results Theorems \ref{mainconv} and \ref{theorem-non-degenerate}. We already know from Lemma \ref{symmprop} that there exists $x_0\in\R$ such that every element in $\omega(u)$ is symmetric with respect to $x=x_0$. In this section and the following section,  let $x_0$ denote this point of symmetry.

\subsection{General nonlinearity}
This section is devoted to the proof of Theorems \ref{mainconv} for problem \eqref{E} with general nonlinearity, that is, $f$ is only assumed to satisfy \eqref{0state}, \eqref{tperiod} and \eqref{assume-lipschitz}. The proof is based on the following key observation. 

\vskip 5pt

\begin{lemma}\label{lem:omega-max}
Let  $\varphi$  be an element of $\omega(u)$ such that
\[
   \varphi(x_0) = \max \big\{ \phi(x_0) \mid \phi\in \omega(u)\big\} 
\]
and let $v_\varphi$ be the $\omega$-limit solution through 
$\varphi$.  Then 
\[
   v_\varphi(x_0,t)  > v(x_0,t) \, \hbox{ for } t\in \R
\]
for any $\omega$-limit solution $v\ne v_\varphi$. Moreover, either of the following holds:
\begin{itemize}
 \item[{\rm (I)}] 
  $\varphi$ is constant, and $v_\varphi=v_\varphi(t)\in \mathcal{X}_{per}$; 
 
\item[{\rm (II)}]   
$v_\varphi(x,t)$ is a 
symmetrically decreasing $T$-periodic solution of \eqref{ground} that 
converges to some $p\in\mathcal{X}_{per}$ as  $x \to \pm\infty$.
\end{itemize}
\end{lemma}

\begin{proof}
The proof of this lemma follows immediately from Lemmas \ref{lem:periodic-tangential}-\ref{omega-sol}.
\end{proof}

\vskip 5pt
In what follows, we divide the proof of Theorem \ref{mainconv} into two cases according to the behavior of $v_{\varphi}$ stated in the above lemma. We first consider the case where $v_\varphi=v_\varphi(t)$ is a spatially homogeneous solution.

\vskip 5pt

\begin{prop}\label{prop:flat}
Suppose that the case {\rm (I)} in Lemma \ref{lem:omega-max} holds. 
Then we have $\omega(u)=\{\varphi\}$. 
Consequently, for each $t>0$,
\begin{equation}\label{prop:flat-con}
u(x,t+mT) \to v_\varphi(t)\, \hbox{ as } m \to\infty \hbox{ in } L^{\infty}_{loc}(\R).
\end{equation}
\end{prop}

\begin{proof}
Suppose that $\omega(u)$ contains another element 
$\tilde{\varphi}$.  Then, by Lemma \ref{lem:omega-max}, 
$\tilde{\varphi}(x_0)<\varphi(x_0)$. Since $\tilde{\varphi}(x)$ is either a 
constant or symmetrically decreasing by Lemma \ref{conorsym}, and since 
$M:=\varphi$ is constant, we have 
$$ \tilde{\varphi}(x) < M \,\,  \hbox{ for } x\in\R. $$
Since $\tilde{\varphi}$ is an $\omega$-limit point, in view of Lemma \ref{monotone}, 
we have 
\begin{equation}\label{ineq1}
u(x,m_1T)\leq M-\delta\,\,  \hbox{ for } x\in\R
\end{equation}
for some integer $m_1>0$ and some constant $\delta>0$. 
Now, choose $L>0$ large enough so that 
\begin{equation}\label{prop:flat-support}
[{\rm spt}(u_0)]\subset [x_0-L,x_0+L]
\end{equation}
and that 
\begin{equation}\label{ineq2}
u(x,t)\leq V_M(t)-\delta \quad\ \hbox{for} 
\ \ |x-x_0|\geq L,\; t\in [0,m_1 T],
\end{equation}
where $V_M(t)$ denotes the $\omega$-limit solution 
through $M$, which is a $T$-periodic function of $t$ 
by Lemma \ref{lem:omega-max}.  Such a constant $L$ 
exists since $u(x,t)$ decays to $0$ as $x\to\pm\infty$ 
uniformly in $t\in [0,m_1 T]$ because of Lemma \ref{largex}. 
Next, since $V_M$ is an $\omega$-limit solution, we 
can find an integer $m_2>0$ such that
\begin{equation}\label{ineq3}
\Big|u(x,t+m_2 T)-V_M(t)\Big|\leq \frac{\delta}{2}\quad\ 
\hbox{for} \ \ |x-x_0|\leq L,\; t\in [0,m_1 T].
\end{equation}
Combining \eqref{ineq1} and \eqref{ineq3}, and 
setting $t=m_1 T$, we get
\begin{equation}\label{ineq4}
u(x,(m_1+m_2)T)> u(x,m_1 T) \quad\ \hbox{for} \ \ 
|x-x_0|\leq L.
\end{equation}
Also, combining \eqref{ineq2} and \eqref{ineq3} 
and setting $x=x_0\pm L$, we obtain
\[
u(x_0\pm L, t+m_2 T)>u(x_0\pm L,t)
 \quad\ \hbox{for} \ \ t\in[0,m_1 T].
\]
This, together with \eqref{prop:flat-support} and the comparison 
principle implies
\[
u(x, t+m_2 T)>u(x,t)
 \quad\ \hbox{for} \ \ |x-x_0|\geq L,\; t\in[0,m_1 T].
\]
Setting $t=m_1 T$ in the above inequality and combining 
it with \eqref{ineq4}, we obtain
\[
u(x,(m_1+m_2)T)> u(x,m_1 T) \quad\ \hbox{for} \ \ 
x\in\R.
\]
Thus, by the comparison principle, we have
\begin{equation*}\label{comparison}
u(x,t+m_2T)> u(x,t) \quad\ \hbox{for} \ \ 
x\in\R,\;t\in[m_1T,\infty).
\end{equation*}
Consequently,
\begin{equation}\label{monotone-1}
u(x,t)<u(x,t+m_2T)<u(x,t+2m_2 T)<u(x,t+3m_2 T)<\cdots
\end{equation}
for $x\in\R,\;t\geq m_1T$. 

Next, let $\{n_j\}_{j\in\N}$ be a sequnce of integers such that 
$$ u(x,n_j T)\to M \,\,\hbox{ as } n_j\to\infty  \,\,\hbox{ in } L_{loc}^{\infty}(\R). $$ 
Since $V_M(t)$ is a $T$-periodic solution, we have
$$ u(x,kT+n_j T)\to M \,\,\hbox{ as } n_j\to\infty  \,\,\hbox{ in } L_{loc}^{\infty}(\R). $$ 
for $k=0, 1,2,\ldots, m_2-1$. Combining this and \eqref{monotone-1}, 
we see that
$$u(x, m T)\to M\,\,\hbox{ as } m\to\infty  \,\,\hbox{ in } L_{loc}^{\infty}(\R). $$ 
This means that $M$ is the only element of $\omega(u)$, 
contradicting the assumption at the beginning of the 
present proof. Thus, $\omega(u)$ cannot contain more 
than one element. The convergence \eqref{prop:flat-con} then 
follows by standard parabolic estimates.  The proposition is proved.
\end{proof}

\vskip 5pt

Next, we turn to the case where $v_\varphi(x,t)$ is symmetrically decreasing. Before stating the conclusion, let us prove two auxiliary lemmas. We begin with the following observation on the intersection numbers of any two $T$-periodic $\omega$-limit solutions.

\vskip 5pt

\begin{lemma}\label{periodic-number} 
Let $v_1$ and $v_2$ be two $T$-periodic $\omega$-limit solutions of \eqref{E}. Then one of the following holds:
\begin{itemize}
 \item[{\rm (i)}] $v_1\equiv v_2$;
 \item[{\rm (ii)}]  $\mathcal{Z}(v_1(\cdot,t)-v_2(\cdot,t))=0$ for each  $t\in\R$;
 \item[{\rm (iii)}]  $\mathcal{Z}(v_1(\cdot,t)-v_2(\cdot,t))=2$ for each  $t\in\R$.
 \end{itemize}
\end{lemma}

\begin{proof}
We assume that $v_1\not\equiv v_2 $, and prove either (ii) or (iii) holds.  
By Lemma \ref{lem:periodic-tangential}, we have $v_1(x_0,t)\neq v_2(x_0,t)$ for each $t\in\R$, 
and 
$$0\leq \mathcal{Z}(v_1(\cdot,t)-v_2(\cdot,t))<\infty \,\,\hbox{ for } t\in\R.  $$
Without loss of generality, we may assume that 
\begin{equation}\label{v1geqv2}
v_1(x_0,t)>v_2(x_0,t)\,\,\hbox{ for } t\in\R. 
\end{equation}
Since $\mathcal{Z}(v_1(\cdot,t)-v_2(\cdot,t))$ is nonincreasing in $t$ and since $v_1(\cdot,t)$ and $v_2(\cdot,t)$ are $T$-periodic, it follows from Lemma \ref{zero1} that $v_1(\cdot,t)-v_2(\cdot,t)$ has a fixed number of zeros for all $t\in\R$, and all of them are simple. 

Assume by contradiction that neither (ii) nor (iii) holds. Then, since $v_1(x,t)$ and $v_2(x,t)$ are both symmetric with respect to $x_0$, we have 
\begin{equation}\label{zeronumber4}
m:=\mathcal{Z}(v_1(\cdot,t)-v_2(\cdot,t)) >2 \,\,\hbox{ for } t\in\R, 
\end{equation}
and $m$ is an even constant. Denote these zeros by 
$$\xi_1(t) <\xi_2(t)<\cdots<\xi_m(t).$$
Clearly, for each $1\leq i \leq m$, $\xi_i(t)$ is a continuous and bounded function of $t\in\R$. Moreover, since $v_1(\cdot,t)-v_2(\cdot,t)$  is $T$-periodic,  $\xi_i(t)$ is $T$-periodic.

Next, by \eqref{v1geqv2} and \eqref{zeronumber4}, we can choose $\xi_{\pm}(t) \subset \{\xi_i(t)\}_{1\leq i\leq m}$ satisfying that, for each $t\in\R$, $-\infty< \xi_-(t) <\xi_+(t)<x_0$, 
$$v_1(\xi_{\pm}(t),t)-v_2(\xi_{\pm}(t),t)=0 \quad\hbox{and}\quad v_1(x,t)-v_2(x,t)<0\,\hbox{ for }  \xi_-(t) <x<\xi_+(t).$$ 
Let $t_0\in\R$ be fixed.  Then we can find $d_0>0$ such that $v_1(\cdot+d_0,t_0)$ and $v_2(\cdot,t_0)$ are tangent at some point $y_0 \in (\xi_-(t_0), \xi_+(t_0))$. More precisely, we have
$$v_2(x,t_0)\leq v_1(x+d_0,t_0) \,\hbox{ for } \,\xi_-(t_0) \leq x\leq \xi_+(t_0),$$
and 
\begin{equation}\label{v2tangentv1}
v_2(y_0,t_0)= v_1(y_0+d_0,t_0).
\end{equation}
Since $v_1(x,t)$ is increasing in $x \leq x_0$, it is clear that  
$$v_2(\xi_{\pm}(t),t) =v_1(\xi_{\pm}(t),t)<v_1(\xi_{\pm}(t)+d_0,t)  \,\hbox{ for } \, t\in\R. $$
Then applying strong maximum principle to the equation satisfied by $v_2(x,t) -v_1(x+d_0,t)$ over the region
$\big\{(x,t):\, \xi_-(t) \leq x\leq \xi_+(t),\, t\geq t_0 \big\}$,
 we have
\begin{equation}\label{v2-leq-v1shift}
v_2(x,t)  <v_1(x+d_0,t)\,\,\hbox{ for }  \xi_-(t) <x <\xi_+(t),\,t>t_0.
\end{equation}
On the other hand, it follows directly from \eqref{v2tangentv1} and the $T$-periodicity of $v_1$ and $v_2$ that 
$$ v_2(y_0,t_0+T)= v_1(y_0+d_0,t_0+T),$$
which is a contradiction with \eqref{v2-leq-v1shift}, since $y_0\in (\xi_-(t_0+T),\xi_+(t_0+T) )$ due to the $T$-periodicity of $\xi_{\pm}(t)$. Therefore, the proof of Lemma \ref{periodic-number} is complete.
\end{proof}

The following lemma provides a uniform estimate (in $t$) for the solution $u(x,t)$ at all large $x$ if Case {\rm (II)} of Lemma \ref{lem:omega-max} happens.

\begin{lemma}\label{shift-supersolu} 
Suppose that Case {\rm (II)} in Lemma \ref{lem:omega-max} holds. Then there exists $L_0>0$ sufficiently large such that 
\begin{equation}\label{u-largex-estimate}
u(x+x_0,t)< v_{\varphi}(|x|-L_0+x_0,t)\,\,\hbox{ for } |x|\geq L_0,\,t\geq 0.
\end{equation}
Furthermore, for any  $v\in \tilde{\omega}(u)$, we have 
\begin{equation}\label{xinfty-omega-estimate}
\lim_{|x|\to\infty } v(x,t) \leq  \lim_{|x|\to\infty } v_{\varphi}(x,t) \,\hbox{ for } t\in\R. 
\end{equation}
\end{lemma}

\begin{proof}
We first show that there exists $L_0>0$ such that 
\begin{equation}\label{ulargex-leq-v0}
u(x,t)< v_{\varphi}(x_0,t) \,\,\hbox{ for  }  |x-x_0|\geq L_0,\,t\geq 0.
\end{equation}
Suppose the contrary that it does not hold. Then there exists a sequence $\{(y_k,t_k)\}_{k\in\N}$ in $\R\times \R^+$ with $|y_k|\to\infty$ as $k\to\infty$ such that
\begin{equation}\label{assume-uk-leq-vk}
u(y_k,t_k)\geq v_{\varphi}(x_0,t_k) \,\,\hbox{ for } k\in\N.
\end{equation} 
In what follows, we shall select various subsequences from $\{(y_k,t_k)\}$ and, to avoid inundation
by subscripts, always denote the subsequence again by $\{(y_k,t_k)\}$.  Without loss of generality, we assume that there is a subsequence of $\{y_k\}$ such that $y_k\to\infty$ as $k\to\infty$.

We claim that the sequence $\{t_k\}$ is unbounded. Otherwise, by Lemma \ref{largex}, we have $u(y_k,t_k)\to 0$
as $k\to\infty$. This is impossible, since \eqref{assume-uk-leq-vk} implies
$$ \liminf_{k\to\infty}u(y_k,t_k) \geq  \liminf_{k\to\infty} v_{\varphi}(x_0,t_k) >0.$$ 
Thus, $\{t_k\}$ is unbounded.

We write $t_k=t_k'+t_k''$ with $t_k'\in T\N$ and $t_k''\in [0,T)$. Then there exist $s_0\in\R$ and an $\omega$-limit point $\tilde{\varphi}$ such that, up to extraction some subsequence, $t_k''\to s_0$ as $k\to\infty$, and 
$$u(x,t_k')\to \tilde{\varphi}(x) \,\hbox{ as } k\to\infty \hbox{ locally uniformly in }  x\in\R. $$
Let $v_{\tilde{\varphi}}$ be the $\omega$-limit solution through $\tilde{\varphi}$. It then follows from Lemma \ref{lem:omega-max} that 
\begin{equation}\label{tvarphiy-leq-varphi0}
\hbox{either } \,v_{\tilde{\varphi}}\equiv v_{\varphi}\, \hbox{ or } \, v_{\tilde{\varphi}}(x_0,t) < v_{\varphi}(x_0,t) \,\hbox{ for all } t\in\R.
\end{equation}

On the other hand, let $z_0>0$ be a number sufficiently large such that $z_0\not\in [{\rm spt}(u_0)]$. 
Then there exists $k_0\in\N$ large enough such that $y_k\geq z_0$ for all $k\geq k_0$. It follows from Lemma \ref{monotone} that 
$$u(z_0,t_k) \geq u(y_k,t_k) \, \hbox{ for all } k\geq k_0.$$
Combining this with \eqref{assume-uk-leq-vk}, we obtain 
$$u(z_0,t_k)\geq v_{\varphi}(x_0,t_k)\, \hbox{ for all } k\geq k_0.$$ 
Since $ v_{\varphi}$ is $T$-periodic, passing to the limit as $k\to\infty$, we have 
$ v_{\tilde{\varphi}}(z_0,s_0)\geq v_{\varphi}(x_0,s_0)$. 
This together with \eqref{tvarphiy-leq-varphi0} and the fact that $v_{\tilde{\varphi}}$ is either symmetrically decreasing with respect to $x_0$ or spatially homogeneous implies that $v_{\tilde{\varphi}}\equiv v_{\varphi}$ and $v_{\tilde{\varphi}}$ is spatially homogeneous. It  contradicts with our assumption that $v_{\varphi}$ is symmetrically decreasing. Therefore, \eqref{ulargex-leq-v0} is proved.

By enlarging $L_0$ in \eqref{ulargex-leq-v0} if necessary, we may assume that 
$[{\rm spt}(u_0)]\subset [x_0-L_0,x_0+L_0]$. We now consider the function 
$$w(x,t):= v_{\varphi}(|x|-L_0+x_0,t) $$
over the region $\Omega:=\{(x,t):\,|x|\geq L_0, t\geq 0 \} $. It is easily checked that $w$ satisfies
\begin{equation*}
\left\{\baa{ll}
\smallskip w_t = w_{xx}+f(t,w), & \hbox{ for } \,|x|\geq L_0,\,t>0, \vspace{3pt}\\
\smallskip w(x,0)>0, & \hbox{ for }\, \,|x|\geq L_0, \vspace{3pt}\\
\smallskip w(\pm L_0,t) > u(x_0\pm L_0,t), & \hbox{ for }\, t\geq 0.\eaa\right.
\end{equation*}
Then, by the strong maximum principle, we have 
$$u(x+x_0,t) < w(x,t) \,\,\hbox{ for }  |x|\geq L_0,\, t\geq 0.$$
Thus, we obtain \eqref{u-largex-estimate}. 

It remains to prove \eqref{xinfty-omega-estimate}. Let $v$ be any $\omega$-limit solution of \eqref{E}. Then there exists a sequence of positive integers $\{m_j\}_{j\in\N}$ such that for any $t\in\R$, 
$$u(x,m_jT+t)\to v(x,t) \,\hbox{ as }  m_j\to\infty \, \hbox{ in }   L_{loc}^{\infty}(\R).  $$
It follows from \eqref{u-largex-estimate} that for any $t\in\R$ and any $j\in\N$ with $t+m_jT>0$, there holds
\begin{equation*}
u(x+x_0,m_jT+t)< v_{\varphi}(|x|-L_0+x_0,m_jT+t)\,\,\hbox{ for } |x|\geq L_0.
\end{equation*}
Taking the limit as $j\to\infty$, we obtain 
\begin{equation*}
v(x+x_0,t)\leq  v_{\varphi}(|x|-L_0+x_0,t)\,\,\hbox{ for } |x|\geq L_0.
\end{equation*}
This in particular implies \eqref{xinfty-omega-estimate}. The proof of Lemma \ref{shift-supersolu}  is thus complete. 
\end{proof}

\vskip 5pt

We are now ready to show the following result.

\begin{prop}\label{prop:nonflat}
Suppose that Case {\rm (II)} in Lemma \ref{lem:omega-max} holds. 
Then either $\omega(u)=\{\varphi\}$ or Case (ii) of Theorem \ref{mainconv} holds.
\end{prop}

\begin{proof}
It is clear that if $\omega(u)$ is a singleton, then $\omega(u) =\{\varphi\}$. In what follows, we assume that $\omega(u)$ contains more than one element, and show the conclusions in Case (ii) of Theorem \ref{mainconv} hold.  

Let $V(x,t)$ be any $T$-periodic $\omega$-limit solution of \eqref{E}. By Lemma \ref{conorsym}, $V(x,t)$ is either symmetrically decreasing with respect to $x_0$ or spatially homogeneous. 
It follows from Lemma \ref{periodic-number} that 
$$\mathcal{Z}\big(v_{\varphi}(\cdot,t)- V(\cdot,t)\big)\equiv 0 \,\hbox{ or }\, 2. $$
We will use the following two notations
$$p_{\varphi}(t):=\lim_{|x|\to\infty} v_{\varphi}(x,t)\quad\hbox{and}\quad  p(t):=\lim_{|x|\to\infty} V(x,t) \,\, \hbox{ for } t\in\R.  $$
Clearly, $p_{\varphi}$ and $p$ are $T$-periodic solutions of \eqref{ODE}.

For clarity, we divide the proof into 3 steps.

{\bf Step 1: } We show that $V(x,t)$ is symmetrically decreasing with respect to $x_0$.

Suppose the contrary that $V=V(t)$ is spatially homogeneous. Then,  if $\mathcal{Z}\big(v_{\varphi}(\cdot,t)- V(t)\big)\equiv 2$, 
it is easily seen that
\begin{equation*}
p_{\varphi}(t) < p(t)=V(t) \,\,\hbox{ for }  t\in\R,
\end{equation*}
which is a contradiction with Lemma \ref{shift-supersolu}. 

On the other hand, if $\mathcal{Z}\big(v_{\varphi}(\cdot,t)- V(t)\big)\equiv 0$, then we have
$$V(t) \leq p_{\varphi}(t)<v_{\varphi}(x_0,t)\,\, \hbox{ for } t\in\R.$$
It follows from Corollary \ref{coro-intermediate} that there exists another $q\in \mathcal{X}_{per}$ such that
$$V(t) \leq p_{\varphi}(t) < q(t)<v_{\varphi}(x_0,t)\,\, \hbox{ for } t\in\R.$$
Since the set $\big\{v(\cdot,t):\,t\in\R,\,v\in\tilde{\omega}(u)\big\}$ is connected and compact in the topology of $L^{\infty}_{loc}(\R)$, there exists $v\in \tilde{\omega}(u)$ such that 
$$v(x_0,t_0)=q(t_0) \,\hbox{ for some } t_0\in\R.  $$ 
Note that $v(x,t)$ cannot be $T$-periodic. Otherwise, Lemma \ref{lem:periodic-tangential} implies $v\equiv q$, and hence, $v(t)> p_{\varphi}(t)$ for $t\in\R$, which is a contradiction with Lemma \ref{shift-supersolu} again. It then follows from Lemma \ref{omega-sol} that $v(x_0,t+mT)$ is strictly monotone in $m\in\Z$, and there exist two $T$-periodic $\omega$-limit solutions $V_{\pm}(x,t)$ that are connected by $v(x,t)$ in the sense of \eqref{omega-lim-pm}. Without loss of generality, we assume that $v(x_0,t+mT)$ is increasing in $m\in\Z$. Then we have
$$V_-(x_0,t_0)< q(t_0)< V_+(x_0,t_0), $$
which is a contradiction with Step 1 of the proof of Proposition \ref{hereconnect}.  Thus, $V(x,t)$ is symmetrically decreasing with respect to $x_0$.

{\bf Step 2:}  we prove $p \equiv p_{\varphi}$.

Suppose the contrary that it is invalid. Then by Lemma \ref{shift-supersolu}, there holds
$p(t) <  p_{\varphi}(t)$ for $t\in\R$.  It is then easily seen that $$V(x,t)< v_{\varphi}(x,t) \,\,\hbox{ for } x\in\R,\,t\in\R,$$
and that either of the following holds:
\begin{itemize}
\item [(a)]  $p(t)<p_{\varphi}(t) < V(x_0,t) <v_{\varphi}(x_0,t)$ for $t\in\R$;
\item [(b)]  $p(t)<V(x_0,t)<  p_{\varphi}(t) <v_{\varphi}(x_0,t)$ for $t\in\R$. 
\end{itemize}

We first rule out Case (a). Assume by contradiction that it holds. Note that both $V(x,t)$ and $v_{\varphi}(x,t)$ are symmetrically decreasing with respect to $x_0$.
For any fixed $t_0\in\R$, one can find a constant $d_0>0$ such that $v_{\varphi}(x,t_0)$ and $V(x+d_0,t_0)$ are tangent at some point $y_0<x_0$. More precisely, we have
$$V(y_0+d_0,t_0)=  v_{\varphi}(y_0,t_0)\quad\hbox{and}\quad V(x+d_0,t_0) \leq   v_{\varphi}(x,t_0) \,\,\hbox{ for } x\in\R.$$
Then by the strong maximum principle, we have 
$$V(x+d_0,t) <  v_{\varphi}(x,t) \,\,\hbox{ for } x\in\R,\,t>t_0,$$
which contradicts $V(y_0+d_0,t_0+T) =  v_{\varphi}(y_0,t_0+T)$. Therefore, Case (a) is impossible. 

Suppose that Case (b) occurs. It follows from Corollary \ref{coro-intermediate} that there exists another $q\in \mathcal{X}_{per}$ such that
$$V(x_0,t)<  p_{\varphi}(t) < q(t)<v_{\varphi}(x_0,t)\,\, \hbox{ for } t\in\R.$$
This is impossible, due to the same reasoning as used in Step 1.  Case (b) is ruled out too. Thus, we obtain $p \equiv p_{\varphi}$.

{\bf Step 3: } Completion of the proof. 

Note that if $\tilde{\omega}(u)$ consists of $T$-periodic solutions, then they are all symmetrically decreasing with respect to $x_0$ and have the same base $p_{\varphi}(t)$, and hence, the proof is complete.  

Suppose, on the other hand, that $\tilde{\omega}(u)$ contains non-periodic elements. Let $w$ be any such element. It follows from Lemma \ref{omega-sol} that $w(x_0,t+mT)$ is strictly monotone in $m\in\Z$, and there exist two $T$-periodic functions $W_{\pm}\in\tilde{\omega}(u)$ such that for any $t\in\R$, 
\begin{equation}\label{v-con-Vpm}
w(x,t+mT)\to W_{\pm}(x,t) \hbox{ as } m\to\pm \infty \hbox{ in } L_{loc}^{\infty}(\R).
\end{equation}
By Step 1, $W_{\pm}$ are symmetrically decreasing with respect to $x_0$. It is clear that $w$ is also symmetrically decreasing with respect to $x_0$. 
Denote 
$$p_{w}(t):=  \lim_{|x|\to\infty}w(x,t) \,\,\hbox{ for } t\in\R.$$
It remains to show that $p_{w}\equiv p_{\varphi}$ and the convergence \eqref{v-con-Vpm} takes place in the topology of $L^{\infty}(\R)$. 

Suppose the contrary that $p_{w}\not\equiv p_{\varphi}$. By Lemma \ref{shift-supersolu}, we have 
$$p_{w}(t)  < p_{\varphi}(t) \hbox { for } t\in\R.$$ 
It is clear that $p_{w}(t)$ is a solution of \eqref{ODE} and that it is either $T$-monotone nonincreasing or $T$-monotone nondecreasing. 
We may assume without loss of generality that $p_{w}(t)$ is $T$-monotone nondecreasing, since the analysis for the other cases is parallel.
Note that by standard parabolic estimates, $w(x,t)$ converges to $p_{w}(t)$ as $|x|\to\infty$ uniformly in $t\in\R$. There exists $M>0$ such that  
\begin{equation*}
w(x,mT)\leq  p_{w}(mT)+\delta  \,\,\hbox{ for all } |x|\geq M,\, m\in\Z, 
\end{equation*} 
where $\delta:= 1/2 (p_{\varphi}(0)-p_{w}(0)) >0$. Since $p_{w}(mT)$ is nondecreasing in $m\in\N$, we have 
$p_{w}(mT) \leq p_{w}(0)$ for all negative integers $m$. Thus, there holds
\begin{equation}\label{tildev-leq-p+}
w(x,mT) \leq p_{\varphi}(0)-\delta  \,\,\hbox{ for all } |x|\geq M\, \hbox{ and all negative integers } m. 
\end{equation}
On the other hand, by \eqref{v-con-Vpm}, for the above $\delta>0$ and $M>0$, there exists $m_0\in\Z$ sufficiently large negative such that 
$$  
w(x,m_0T)> W_{-}(0,x)-\delta \,\,\hbox{ for all }  |x|\leq M.
$$
Since $W_{-}(t,x)> p_{\varphi}(t)$ for $x\in\R$, $t\in\R$ by Step 2, we have
$$  
w(x,m_0T) > p_{\varphi}(0)-\delta \,\,\hbox{ for all }  |x|\leq M,
$$
which is a contradiction with \eqref{tildev-leq-p+} by choosing $|x|=M$ and $m=m_0$. Therefore, $p_{w}\equiv p_{\varphi}$, that is, $w(x,t)$ converges to $p_{\varphi}(t)$ as $|x|\to\infty$ uniformly in $t\in\R$. This together with the fact that $\lim_{|x|\to\infty} W_{\pm}(x,t) =p_{\varphi}(t) $ uniformly in $t\in\R$, immediately implies that the convergence \eqref{v-con-Vpm} takes place in the topology of $L^{\infty}(\R)$. 
The proof of Proposition \ref{prop:nonflat} is thus complete. 
\end{proof}

\vskip 5pt

\begin{proof}[Proof of Theorem \ref{mainconv}]
The proof follows directly from Lemma \ref{lem:omega-max}, Proposition \ref{prop:flat} and  Proposition \ref{prop:nonflat}.
\end{proof}

\subsection{Non-degenerate nonlinearity} 
In this subsection, we give the proof of Theorem \ref{theorem-non-degenerate}. Namely, we exclude the case (ii) of Theorem \ref{mainconv} under the additional non-degenerate assumption (H).  
\vskip 5pt

\begin{proof}[Proof of Theorem \ref{theorem-non-degenerate}]
Let $V_{\pm}(x,t)$ be any two symmetrically decreasing $T$-periodic $\omega$-limit solutions of \eqref{E} based at some $p\in \mathcal{X}_{per}$. It follows from Proposition \ref{nonexistence-result} that $p\notin \mathcal{Y}_{per}$. Then by the assumption (H), $p(t)$ is linearly stable. To complete the proof, it suffices to show that
\begin{equation}\label{v-equiv-v+}
V_{-}\equiv V_{+}. 
\end{equation}
This proof is similar in spirit to \cite[Proposition 3.2]{fp}, but nontrivial changes are needed.  For the sake of completeness, we include the details as follows.  

 It is clear that
$$\mu:= -\int_{0}^{T} \partial_u f(t,p(t))dt > 0.$$
Let $a>0$ be a fixed constant and $\bar{h}\in C^1([0,\infty))$ be the solution of the linear ODE
\begin{equation*}
\bar{h}_t  =\Big(\partial_u f(t,p(t))+\frac{\mu}{2}\Big)\bar{h} \,\,\hbox{ for } t>0,\quad\quad \bar{h}(0)=a. 
\end{equation*}
It is straightforward to check that $\bar{h}(t)>0$ for $t\geq 0$ and 
\begin{equation}\label{asymt-stable}
\bar{h}(t) \to 0 \,\,\hbox{ as } t\to\infty. 
\end{equation}

We now assume by contradiction that \eqref{v-equiv-v+} does not hold. Then we have $V_{+}(x_0,t) \neq V_-(x_0,t)$ for all $t\in\R$. Otherwise, \eqref{v-equiv-v+} follows immediately from Lemma \ref{lem:periodic-tangential}. Without loss of generality, we assume that 
\begin{equation}\label{v1-x0-geq-v2}
V_+(x_0,t) > V_-(x_0,t) \,\,\hbox{ for } t\in\R. 
\end{equation}
 For any $\lambda\in\R$, set 
 $$W^{\lambda}(x,t):= V_-(x,t)-V_{+}(x-\lambda,t)\,\,\hbox{ for }  x\in\R,\,t\in\R. $$
Clearly, $W^{\lambda}(x,t)$ is $T$-periodic and nonzero. 
By the proof of Lemma \ref{periodic-number},  for each $\lambda\in\R$, $W^{\lambda}(\cdot,t)$ has a fixed number of zeros, and all of them are simple. It follows from \eqref{v1-x0-geq-v2} that for all large negative $\lambda$, $W^{\lambda}(x,t)$ has zeros to the right of $x_0+\lambda$. Denote the minimum of such zeros by $\xi^{\lambda}(t)$. Then $\xi^{\lambda}(t)$ is continuous and $T$-periodic. Moreover, there exists $\lambda_1\in (-\infty,\infty]$ such that $\xi^{\lambda}(0)$ is well-defined in $\lambda\in (-\infty,\lambda_1)$ and 
\begin{equation}\label{converge-xi0}
\xi^{\lambda}(0) \to \infty \,\,\hbox{ as } \lambda\to \lambda_1. 
\end{equation}
Next, we show that
\begin{equation}\label{converge-xi}
\xi^{\lambda}(t) \to \infty \,\,\hbox{ as } \lambda\to \lambda_1 \,\hbox{ uniformly in } t\in\R. 
\end{equation}
Indeed, since $\xi^{\lambda}(t) > x_0+\lambda$ for each $\lambda\in (-\infty,\lambda_1)$ and $t\in\R$, it immediately gives \eqref{converge-xi} if $\lambda_1=\infty$. On the other hand, if $\lambda_1<\infty$, then by \eqref{converge-xi0} and the structure of $\xi^{\lambda}(t)$, we obtain  
$$W^{\lambda_1} (x,t) >0 \,\,\hbox{ for all } x\geq x_0+\lambda_1,\,t\in\R,$$
which also implies \eqref{converge-xi}. Thus, \eqref{converge-xi} is proved. 

We choose a small constant $\delta\in (0,a)$ such that 
\begin{equation}\label{choose-small-delta}
\partial_u f(t,s) \leq  \partial_u f(t,p(t))+\frac{\mu}{2}\,\, \hbox{ for all } p(t)\leq s\leq p(t)+\delta,\,t\in\R. 
\end{equation}
By \eqref{converge-xi}, there exists $\lambda_{\delta}\in (-\infty,\lambda_1)$ sufficiently close to $\lambda_1$ such that 
$$p(t)<V_-(x,t) \leq p(t)+\delta\,\,\hbox{ for }  x\geq \xi^{\lambda_{\delta}}(t),\,t\in\R.   $$
Since $\xi^{\lambda_{\delta}}(t) > x_0+\lambda_{\delta}$ for $t\in\R$ and since $V_{+}(x-\lambda_{\delta},t) $ is decreasing in $x>x_0+\lambda_{\delta}$, we have 
$$p(t)<V_+(x-\lambda_{\delta},t) \leq V_+(\xi^{\lambda_{\delta}}(t)-\lambda_{\delta},t) = V_-(\xi^{\lambda_{\delta}}(t),t) \leq p(t)+\delta    $$
for  $x\geq \xi^{\lambda_{\delta}}(t)$, $t\in\R$.

We are now ready to complete the proof. If $W^{\lambda_{\delta}}(x,t)$ has zeros to the right of 
$\xi^{\lambda_{\delta}}(t)$, we denote the minimum of such zeros by $\bar{\xi}(t)$. Clearly, $\bar{\xi}(t)$ is a $T$-periodic and continuous function, and it is also easily checked that
$$  0< W^{\lambda_{\delta}}(x,t) <\delta \,\,\hbox{ for }   \xi^{\lambda_{\delta}}(t) <x< \bar{\xi}(t),\,t\in\R.  $$
Then by \eqref{choose-small-delta}, we calculate that 
\begin{equation*}
 W^{\lambda_{\delta}}_t- W^{\lambda_{\delta}}_{xx}-  \Big(\partial_u f(t,p(t))+\frac{\mu}{2}\Big) W^{\lambda_{\delta}}\leq 0\,\,\hbox{ for }   \xi^{\lambda_{\delta}}(t) <x< \bar{\xi}(t),\,t\in\R.
 \end{equation*}
Moreover, it is clear that 
$$ W^{\lambda_{\delta}}(\xi^{\lambda_{\delta}}(t),t)=0<\bar{h}(t)\quad\hbox{and}\quad W^{\lambda_{\delta}}(\bar{\xi}(t),t)=0<\bar{h}(t) \,\,\hbox{ for } t\geq 0. $$ 
 By the comparison principle, we obtain 
$$\bar{h}(t) \geq   W^{\lambda_{\delta}}(x,t) \,\,\,\hbox{ for }   \xi^{\lambda_{\delta}}(t) \leq x\leq \bar{\xi}(t),\,t\geq 0.$$
It then further follows from \eqref{asymt-stable} that 
$$\lim_{t\to\infty}\sup_{ \xi^{\lambda_{\delta}}(t) \leq x\leq \bar{\xi}(t)} W^{\lambda_{\delta}}(x,t) =0, $$  
which is a contradiction with the fact that $W^{\lambda_{\delta}}(x,t)$ is positive and $T$-periodic on   $\xi^{\lambda_{\delta}}(t) <x< \bar{\xi}(t),\,t\in\R$. On the other hand, if there are no zeros of $W^{\lambda_{\delta}}(x,t)$ to the right of $\xi^{\lambda_{\delta}}(t)$, we can still derive a contradiction by replacing $\bar{\xi}(t)$ by $\infty$ in the above arguments. Therefore, there much hold $V_{+}\equiv V_-$. The proof of Theorem \ref{theorem-non-degenerate} is thus complete. 
\end{proof}


\SE{Proof of Theorem \ref{precise-omega-limit} and Propositions \ref{bistable-theorem} and \ref{combustion-theorem}}
In this section, we first give the proof of Theorem \ref{precise-omega-limit}, that is, any $p\in\mathcal{X}_{per}$ satisfying the condition (i) or (ii) can not belong to $\tilde{\omega}(u)$. The proof relies on Proposition \ref{existence-bound-interval} and the method of change of variables introduced in Section 3. As an easy application of Theorem \ref{precise-omega-limit}, we then prove Propositions \ref{bistable-theorem} and \ref{combustion-theorem}. 

\vskip 5pt

\begin{proof}[Proof of Theorem \ref{precise-omega-limit}]
We assume that $p\in \mathcal{X}_{per}$ is an element of $\tilde{\omega}(u)$, and prove none of the conditions (i) and (ii) can hold. By Corollary \ref{coro-homo}, we immediately obtain 
$\tilde{\omega}(u)=\{p\}$. Then for each $t\geq 0$, 
\begin{equation}\label{ukj-qlambda0}
u(x,t+mT)\to p(t) \,\,\hbox{ as } m\to\infty \hbox{ in } L_{loc}^{\infty}(\R).
\end{equation}

{\bf Step 1:} We show that the condition (i) does not hold.

Suppose the contrary that (i) holds for some $\epsilon>0$, that is, $p(t)$ is unstable from below with respect to \eqref{odeinitial}. Then there exists another $q\in\mathcal{X}_{per}$ satisfying $0\leq q(t)<p(t)$ for $t\in\R$ and, for each $a\in (q(0),p(0))$, 
\begin{equation}\label{assume-unstable-below}
h(t+mT;a) \searrow q(t) \,\,\hbox{ as } m\to\infty \hbox{ uniformly in } t\in [0,T].
\end{equation}

We first show that 
\begin{equation}\label{maxu-leq-p}
\sup_{x\in\R} u(x,t)>p(t) \,\,\hbox{ for all } t\geq 0.
\end{equation}
Assume by contradiction that there exits some $t_0\geq 0$ such that $u(x,t_0)\leq p(t_0)$ for $x\in\R$. Then by strong maximum principle, we have $u(x,t)<p(t)$ for $x\in\R,\,t>t_0$.  Let $a\in (q(0),p(0))$ be sufficiently close to $p(0)$. It follows from \eqref{assume-unstable-below} that
$$\lim_{m\to\infty} u(x,t+mT) \leq \lim_{m\to\infty} h(t+mT;a)  = q(t), $$ 
which is a contradiction with \eqref{ukj-qlambda0}. Hence, we obtain \eqref{maxu-leq-p}. 

By Lemma \ref{zero1}, we have $\mathcal{Z}(u(\cdot,t)-p(t))<\infty$ for $t>0$ and it is a constant for large $t$. Due to Lemma \ref{largex} and \eqref{maxu-leq-p}, this constant is positive. Therefore, for all large $t$, say $t\geq T_1$,  $u(x,t)-p(t)$ has a fixed number of zeros and all of them are simple. Denote the maximum of these numbers by $\xi(t)$ for $t\geq T_1$. Clearly, $\xi(t)$ is a continuous function of $t$, and there holds
\begin{equation}\label{maximum-zero}
u(\xi(t),t)-p(t)=0 ,\quad \partial_x u(\xi(t),t)<0 \,\,\hbox{ for } t\geq T_1.
\end{equation}

Next, we claim that 
\begin{equation}\label{uxxi-to-infty}
\partial_x u(\xi(t),t)\to 0 \,\,\hbox{ as } t\to\infty.
\end{equation}
Otherwise, there would exist a constant $\delta>0$ and a sequence $\{t_k\}_{k\in\N} \subset [T_1, \infty)$ satisfying $t_k\to\infty$ as $k\to\infty$ and 
\begin{equation}\label{ux-les-sigma}
\partial_x u\big(\xi(t_k),t_k\big)\leq -\delta \,\,\hbox{ for each } k\in\N.  
\end{equation}
By standard parabolic estimates, $\partial_{xx} u(x,t)$ is uniformly bounded for $x\in\R$, $t\geq T_1$. This together with the first equality of \eqref{maximum-zero} and \eqref{ux-les-sigma} infers that there exist $r>0$ and $\epsilon>0$ (both independent of $k$) such that 
$$u(\xi(t_k)-r, t_k)-p(t_k)\geq \epsilon \,\,\hbox{ for each } k\in\N.   $$
In view of this and Lemma \ref{monotone},  we can find some $y_0\in {\rm spt} (u_0) $ such that for each $k\in\N$, 
$u(y_0, t_k)-p(t_k)\geq \epsilon$. And hence, we have 
$$\liminf_{k\to\infty} \big(u(y_0, t_k)-p(t_k)\big)\geq \epsilon.$$
This appears to be a contradiction with \eqref{ukj-qlambda0}.  Therefore, \eqref{uxxi-to-infty} holds. 

We complete the proof of this case by constructing a suitable supersolution of \eqref{E} in a finite interval. Because of \eqref{assume-unstable-below}, it follows from Proposition \ref{existence-bound-interval} (see also Remark \ref{re-steady-state}) that, there exists a constant $R>0$ such that the following problem
\begin{equation*}
\left\{\baa{ll}
\smallskip  w_t=w_{xx}+f(t,w), & \hbox{ for } \,-R<x<R,\,t\in\R, \vspace{3pt}\\
\smallskip w(x, t+T)=w(x,T), & \hbox{ for }\,-R\leq x\leq R,\,t\in\R, \vspace{3pt}\\
w(\pm R, t)=p(t), & \hbox{ for }\, t\in \R,\eaa\right.
\end{equation*}
admits a solution $w(x,t)$ satisfying 
\begin{equation}\label{q-leqw-leq-p}
q(t) <w(x,t) < p(t) \,\,\hbox{ for } -R<x<R,\,t\in\R.  
\end{equation}
Clearly, $\partial_x w(-R,t)<0$ for $t\in\R$ and it is $T$-periodic in $t$.  Combining this with \eqref{uxxi-to-infty}, we can find some $T_2\geq T_1$ sufficiently large such that 
\begin{equation}\label{ux-geq-wx}
0> \partial_x u\big(\xi(t),t\big) \geq \frac{1}{2} \partial_x w(-R,t) \,\,\hbox{ for } t\geq T_2. 
\end{equation}
Moreover, note that $u(x,T_2)\to 0$ as $x\to\infty$ by Lemma \ref{largex}. This together with \eqref{q-leqw-leq-p} implies there exists $L>\xi(T_2)$ such that 
$$u(x,T_2)<  w(x-L,T_2) \,\, \hbox{ for } L-R\leq x\leq L+R. $$

For clarity, we divide the following analysis into two cases, based on whether $\xi(t)$ touches $L-R$ at a finite time. Firstly, if $\xi(t) < L-R $ for all $t\geq T_2$, then we have 
$$ u(L\pm R,t) < w(\pm R, t) \,\,\hbox{ for all }  t\geq T_2. $$
Applying the strong maximum principle to the equation satisfied by $u(x,t)-w(x-L,t)$ over $L-R\leq x \leq L+R$, $t\geq T_2$, 
we obtain 
$$u(x,t) < w(x-L,t) \,\,\hbox{ for } L-R\leq x \leq L+R, \, t\geq T_2.$$
This in particular implies for each $t\geq 0$, 
$$\lim_{m\to\infty} u(x,t+mT) \leq w(x-L,t) <p(t) \,\,  \hbox{ for } L-R<x< L+R,   $$
which is a contradiction with \eqref{ukj-qlambda0}.
 
Secondly, if there exists $T_3>T_2$ such that $\xi(T_3) = L-R$ and $\xi(t)<  L-R $ for $T_2\leq t <T_3$, 
then we have 
$$u(L-R,T_3)= w(-R,T_3)\quad\hbox{and}\quad  u(L\pm R,t)< w(\pm R,t)\,  \hbox{ for } T_2\leq t < T_3. $$
By the strong maximum principle again, we obtain   
$$u(x,t) < w(x-L,t) \,\,\hbox{ for } L-R<  x \leq L+R, \, T_2\leq t \leq T_3,$$
and 
$$\partial_x w_x(-R,T_3)> \partial_x u_x(L-R,T_3),$$
which is a contradiction with \eqref{ux-geq-wx}. Therefore, if $p\in \tilde{\omega}(u)$, then the condition (i) does not hold.

{\bf Step 2:} We show that the condition (ii) does not hold. 

Suppose also by contrary that (ii) holds for some $\epsilon>0$. By Step 1, $h(t;a)$ can not be $T$-monotone decreasing for all $a\in (p(0)-\epsilon,p(0))$. Then there exists a positive $\tilde{q}\in\mathcal{X}_{per}$ such that $$p(0)-\epsilon \leq \tilde{q}(0) < p(0).$$ 
Because of \eqref{ukj-qlambda0} and Lemma \ref{monotone}, there exists $m_0\in\N$ sufficiently large such that 
\begin{equation*}
0< u(x,mT) < p(0)+\epsilon\,\,\hbox{ for } x\in\R,\,m\geq m_0. 
\end{equation*}

We may assume without loss of generality that $f(\cdot,u)$ is a $C^2$ function in $u$, as by standard approximation arguments (similar to those given in Step 3 of the proof of Proposition \ref{nonexistence-result}), the following analysis can be extended to the general case where $f$ satisfies \eqref{assume-lipschitz}. Then, for each $k\in\N$, there exsits a unique function $v_k\in C^{2,1}(\R\times [0,T])$ such that 
\begin{equation}\label{change-variable-u-vk}
u(x,m_0T+kT+t)=h(t;v_k(x,t))\,\,\hbox{ for } x\in\R,\,t\in [0,T].
\end{equation}
Clearly, for each $k\in\N$, we have
\begin{equation}\label{prop2-vk}
0< v_k(x,0) < p(0)+\epsilon\,\,\hbox{ for } x\in\R, 
\end{equation}
and
\begin{equation}\label{prop3-vk}
v_{k+1}(x,0)=h(T,v_k(x,T))\,\,\hbox{ for } x\in\R.
\end{equation}
Moreover, since $u(x,t)$ is the solution of \eqref{E}, direct calculation shows that for each $k\in\N$, $v_k(x,t)$ satisfies
\begin{equation}\label{equation-vk}
(v_k)_t=(v_k)_{xx}+ \frac{h_{aa}(t;v_k)}{h_a(t;v_k)}(v_k)_x^2\,\,\hbox{ for } x\in\R,\,t\in [0,T],
\end{equation}
where $h_a(t;a)$ and $h_{aa}(t;a)$ stand for the derivatives of $h(t;a)$ with respect to initial value $a$.
By the estimates proved in Lemma \ref{derivative-estimates}, we have 
$$ \frac{h_{aa}(t;v_k)}{h_a(t;v_k)} \leq M\,\,\hbox{ for } x\in\R,\,t\in [0,T],\,k\in\N,$$
where $M$ is a positive constant independent of $k$. 

Let $\psi(x,t)$ be the solution of 
\begin{equation}\label{supequation-psi}
\left\{\baa{ll}
\smallskip  \psi_t=\psi_{xx}+M\psi_x^2, & \hbox{ for } \,x\in\R,\,t>0, \vspace{3pt}\\
\psi(x, 0)=\psi_0(x), & \hbox{ for }\, x\in\R,\eaa\right.
\end{equation}
which initial function $\psi_0\in C(\R)$ satisfying  
\begin{equation}\label{initial-condi1-psi0}
\max\big\{\tilde{q}(0), v_0(x,0)\big\} \leq  \psi_0 (x) \leq p(0)+\epsilon \,\,\hbox{ for } x\in\R,
\end{equation}
and 
\begin{equation}\label{initial-condi2-psi0}
\psi_0 (x)=\tilde{q}(0) \hbox{ for sufficiently large negative and positive } x\in\R. 
\end{equation}
The existence of such a function $\psi_0$ follows from \eqref{prop2-vk} and the fact that $\lim_{|x|\to\infty} v_0(x,0)=0$.

Now we claim that, for each $k\in\N$,
\begin{equation}\label{psi-geq-vk}
v_k(x,t)\leq  \psi(x,t+kT)\,\,\hbox{ for } x\in\R,\,t\in[0,T].
\end{equation}
Indeed, applying the comparison principle to equations \eqref{equation-vk} and \eqref{supequation-psi}, we immediately obtain \eqref{psi-geq-vk} in the case of $k=0$. This in particular implies $v_0(x,T)\leq \psi(x,T)$ for $x\in\R$. Moreover, since $\psi_0$ satisfies \eqref{initial-condi1-psi0}, a simple comparison argument immediately gives
$$\tilde{q}(0)\leq \psi(x,t)\leq  p(0)+\epsilon \,\, \hbox{ for }  x\in\R,\,t\geq 0.$$
Since $h(t;a)$ is $T$-monotone nonincreasing for each $a\in (\tilde{q}(0),p(0)+\epsilon)$, it then follows from 
\eqref{prop3-vk} that 
$$v_1(x,0)=h(T,v_0(x,T)) \leq h(T, \psi(x,T))  \leq  \psi(x,T) \,\,\hbox{ for } x\in\R.$$
By the comparison principle applied to equations \eqref{equation-vk} and \eqref{supequation-psi} again, we obtain 
\eqref{psi-geq-vk} in the case of $k=1$. Then a standard induction argument implies \eqref{psi-geq-vk} holds for all $k\in\N$.

Next, we prove that 
\begin{equation}\label{tinfty-wto0}
\lim_{t\to\infty} \psi(x,t) =\tilde{q}(0) \,\,\hbox{ in } L^{\infty}(\R). 
\end{equation}
By the exponential transformation used the proof of Lemma \ref{estimate-trans-heat}, one easily checks that $\psi(x,t)$ has the following explicit expression
$$\psi(x,t)=\frac{1}{M}\ln \int_{-\infty}^{\infty} \exp\big(M\psi_0(y) \big)K(t,x-y) dy \,\,\hbox{ for } x\in\R,\,t\geq 0,$$
where $K$ is the fundamental solution of the heat equation. Since the function
$\exp\big(M\psi_0(x) \big) -\exp\big(M\tilde{q}(0)\big)$ has a compact support in $\R$ (due to \eqref{initial-condi2-psi0}), it is well known that 
$$\lim_{t\to\infty}\int_{-\infty}^{\infty}\Big[\exp\big(M\psi_0(x) \big) -\exp\big(M\tilde{q}(0) \big) \Big]K(t,x-y) dy =0  \,\,\hbox{ in } L^{\infty}(\R), 
$$
which immediately gives \eqref{tinfty-wto0}.

Finally, combining \eqref{psi-geq-vk} and \eqref{tinfty-wto0}, we obtain
$$\limsup_{k\to\infty}v_k(x,0)\leq \tilde{q}(0)  \,\,\hbox{ for all } x\in\R.$$
On the other hand, it is easily seen from \eqref{ukj-qlambda0} and \eqref{change-variable-u-vk} that
\begin{equation*}
v_k(x,0)\to p(0) \,\,\hbox{ as } k\to\infty \hbox{  in } L_{loc}^{\infty}(\R), 
\end{equation*}
which is a contradiction with $\tilde{q}(0)<p(0)$. Therefore, if $p\in \tilde{\omega}(u)$, then the condition (ii) does not hold.  The proof of Theorem \ref{precise-omega-limit} is thus complete. 
\end{proof}

\vskip 5pt

Next we give the proof of Propositions \ref{bistable-theorem} and \ref{combustion-theorem}. 

\begin{proof}[Proof of Proposition \ref{bistable-theorem}]
In this bistable case, it is clear that $q_1$ is unstable from both above and below with respect to \eqref{odeinitial}. And hence, as an easy application of Proposition \ref{nonexistence-result}, there is no symmetrically decreasing $T$-periodic solution of $\eqref{ground}$ based at $q_{1}$. Note that there is also no such solution based on $p_1$. Otherwise, by Corollary \ref{coro-intermediate}, there would exist an element of $\mathcal{X}_{per}$ above $p_1$, which is a contradiction with our assumption. 
Then, under the additional assumption \eqref{linearly-stable}, it follows from Theorem \ref{theorem-non-degenerate} that $u(x,t+mT)$ converges to a $T$-periodic solution of \eqref{ground} as $m\to\infty$, and the limit can only be:
$$0,\quad U(x,t),\quad q_1(t)\quad \hbox{or}\quad p_1(t).$$
Furthermore, since $q_1$ is unstable from below with respect to \eqref{odeinitial}, Theorem  \ref{precise-omega-limit} implies $q_1\notin \tilde{\omega}(u)$. The proof of Proposition \ref{bistable-theorem} is thus complete.
\end{proof}

\vskip 5pt

\begin{proof}[Proof of Proposition \ref{combustion-theorem}]
In this combustion case, $T$-periodic solutions of \eqref{ODE} are the following:
$$ q_{\lambda}(t)\, \hbox{ for } \,0\leq\lambda\leq 1  \quad\hbox{and}\quad   p_1(t) .$$
It follows directly from Proposition \ref{nonexistence-result} and Corollary \ref{coro-intermediate} that there are no symmetrically decreasing solutions of $\eqref{ground}$ based at the above spatially homogeneous solutions. Then, by Corollary \ref{coro-homo}, $u(x,t+mT)$ converges to one of these $T$-periodic functions as $m\to\infty$. Moreover, 
Theorem \ref{precise-omega-limit} implies $q_{\lambda}(t)$ for $0< \lambda< 1$ can never belong to $\tilde{\omega}(u)$. Therefore, we obtain
$$
\tilde{\omega}(u)= \big\{0\big\},\,\,\big\{q_1\big\} \,\hbox{ or }\, \big\{ p_1\big\}.
 $$
 The proof of Proposition \ref{combustion-theorem} is complete.
\end{proof}

\SE*{Acknowledgements}
The authors would like to thank Prof Peter Pol{\'a}{\v c}ik for many helpful discussions.


\end{document}